\documentclass[a4paper,12pt,reqno]{amsart}
\usepackage{amsmath,amssymb,amsfonts,amsthm,dsfont}% Symbols
\usepackage[utf8]{inputenc} 
\usepackage[english]{babel}
\usepackage[T1]{fontenc} 
\usepackage{graphicx}
\usepackage{float}
\usepackage{enumerate}
\usepackage{marginnote}
\usepackage{hyperref}	
\hypersetup{pdfborder=0 0 0, 
	    colorlinks=true,
	    citecolor=blue,
	    linkcolor=blue,
	    urlcolor=red,
	    pdfauthor={Johannes Sjoestrand, Martin Vogel}
	   }

\newcommand{\mO}{\mathcal{O}}
\newcommand{\C}{\mathds{C}}
\newcommand{\R}{\mathds{R}}
\newcommand{\Z}{\mathds{Z}}
\newcommand{\N}{\mathds{N}}
\newcommand{\Proba}{\mathds{P}}
\newcommand{\E}{\mathds{E}}

\newtheorem{thm}{Theorem}[section]

\newtheorem{prop}[thm]{Proposition}

\theoremstyle{definition}

\theoremstyle{remark}

\theoremstyle{definition}

\theoremstyle{definition}

\theoremstyle{definition}

\numberwithin{equation}{section} 
\title{Interior eigenvalue density of Jordan matrices with random
  perturbations}
\author{Johannes Sj\"ostrand} 
\address[Johannes Sj\"ostrand]{Institut de Mathématiques de Bourgogne - UMR 5584 CNRS, Universit\'e de Bourgogne, 
		       Faculté des Sciences Mirande, 9 avenue Alain Savary, 
		       BP 47870 21078 Dijon Cedex.}
\email{johannes.sjostrand@u-bourgogne.fr}
\author{Martin Vogel} 
\address[Martin Vogel]{Institut de Mathématiques de Bourgogne - UMR 5584 CNRS, Universit\'e de Bourgogne, 
		       Faculté des Sciences Mirande, 9 avenue Alain Savary, 
		       BP 47870 21078 Dijon Cedex.}
\email{martin.vogel@u-bourgogne.fr}
\begin{document}
\dedicatory{Dedicated to the memory of Mikael Passare}
\begin{abstract}
We study the eigenvalue distribution of a large Jordan block 
subject to a small random Gaussian perturbation. A result 
by E.B. Davies and M. Hager shows that as the dimension of the 
matrix gets large, with probability close to $1$, most of the 
eigenvalues are close to a circle. 
\par
We study the expected eigenvalue density of the perturbed Jordan 
block in the interior of that circle and give a precise asymptotic 
description.
  \vskip.5cm
  \par\noindent \textsc{R{\'e}sum{\'e}.}
Nous \'etudions la distribution de valeurs propres d'un grand 
bloc de Jordan soumis \`a une petite perturbation gaussienne 
al\'eatoire. Un r\'esultat de E.B. Davies et M. Hager montre que 
quand la dimension de la matrice devient grande, alors avec 
probabilit\'e proche de $1$, la plupart des valeurs propres sont proches 
d'un cercle.
\par
Nous \'etudions la r\'epartitions moyenne des valeurs propres 
\`a l'int\'erieur de ce cercle et nous en donnons une description 
asymptotique pr\'ecise. 
\end{abstract}
\maketitle
\setcounter{tocdepth}{1}
\tableofcontents
\section{Introduction}
In recent years there has been a renewed interested in the spectral 
theory of non-self-adjoint operators where, as opposed to the self-adjoint 
case, the norm of the resolvent can be very large even far away from 
the spectrum. Equivalently the spectrum of such operators can be 
highly unstable even under very small perturbations of the operator. 
\par
Emphasized by the works of L.N. Trefethen and M. Embree, see for example 
\cite{TrEm05}, E.B. Davies, M. Zworski and many others 
\cite{Da97,Da99,NSjZw04,ZwChrist10,DaHa09}, 
the phenomenon of spectral instability of non-self-adjoint operators has become 
a popular and vital subject of study. In view of this it is very natural to add 
small random perturbations.
\par
One line of recent research concerns the case of elliptic (pseudo)\-differential 
operators subject to small random perturbations, cf. 
\cite{BM,Ha06,Ha06b,HaSj08,SjAX1002,Vo14}.
\subsection{Perturbations of Jordan blocks}
In this paper we shall study the spectrum of a random perturbation of
the large Jordan block $A_0$ :
\begin{equation}\label{jb}
A_0=\begin{pmatrix}0 &1 &0 &0 &...&0\\
                   0 &0 &1 &0 &...&0\\
                   0 &0 &0 &1 &...&0\\
                   . &. &. &. &...&.\\
                   0 &0 &0 &0 &...&1\\
                   0 &0 &0 &0 &...&0
\end{pmatrix}: {\C}^N\to {\C}^N.
\end{equation}
Perturbations of a large Jordan block have already been studied, 
cf. \cite{SjZw07,Zw02,DaHa09,GuMaZe14}.
\begin{itemize}
\item M.~Zworski \cite{Zw02} noticed that for every $z\in D(0,1)$,
  there are associated exponentially accurate quasi-modes when $N\to
  \infty $. Hence the open unit disc is a region of spectral 
  instability.
\item We have spectral stability (a good resolvent estimate) in
${\C}\setminus \overline{D(0,1)}$, since $\| A_0\| =1$.
\item $\sigma (A_0)=\{0 \}$.  
\end{itemize}
Thus, if $A_\delta =A_0+\delta Q$ is a small (random) perturbation of
$A_0$ we expect the eigenvalues to move inside a small neighborhood of $\overline{D(0,1)}$.
\par
In the special case when $Qu=(u|e_1)e_N$, where $(e_j)_1^N$ is the canonical
basis in ${\C}^N$, the eigenvalues of $A_\delta $ are of the form 
$$
\delta ^{1/N}e^{2\pi ik/N},\ k\in {\Z}/N{\Z},
$$
so if we fix $0<\delta \ll 1$ and let $N\to \infty $, the spectrum
``will converge to a uniform distribution on $S^1$''.
\par E.B.~Davies and M.~Hager \cite{DaHa09} studied random
perturbations of $A_0$. They showed that with probability close to 1,
most of the eigenvalues are close to a circle:
\begin{thm}\label{th1}
Let $A=A_0+\delta Q$, $Q=(q_{j,k}(\omega ))$ where $q_{j,k}$ are
independent and identically distributed random variables 
$\sim {\mathcal{N}}_{{\C}}(0,1)$.
If $0<\delta \le N^{-7}$, $R=\delta
^{1/N}$, $\sigma >0$, then with probability $\ge 1-2N^{-2}$, we have
$\sigma (A_\delta )\subset D(0,RN^{3/N})$ and
$$
\# (\sigma (A_\delta )\cap D(0,Re^{-\sigma }))\le \frac{2}{\sigma
}+\frac{4}{\sigma }\ln N.
$$
\end{thm}
  A recent result by A.~Guionnet, P.~Matched Wood and
  O.~Zeitouni \cite{GuMaZe14} implies that when $\delta $ is
  bounded from above by $N^{-\kappa -1/2}$ for some $\kappa >0$ and
  from below by some negative power of $N$, then 
$$
\frac{1}{N}\sum_{\mu \in \sigma (A_{\delta})}\delta (z-\mu )\to \hbox{the
  uniform measure on }S^1,
$$
weakly in probability.
 \\
 \par
  The main purpose of this paper is to obtain, for a small coupling 
  constant $\delta$, more information about the distribution of 
  eigenvalues of $A_\delta$ in the interior of a disc, where the 
  result of Davies and Hager only yields a logarithmic upper bound 
  on the number of eigenvalues; see Theorem \ref{ed5} below.
  \par
  In order to obtain more information in this
  region, we will study the expected eigenvalue density, adapting the
  approach of \cite{Vo14}. (For random polynomials and Gaussian analytic 
  functions such results are more classical, 
  \cite{Kac43,SZ03,HoKrPeVi09,So00,ShZe98,Sh08}.) 
  \\
  \par 
  \textit{Acknowledgments.}
We would like to thank St\'ephane Nonnenmacher for his observation 
on the relation of the density in Theorem \ref{ed5} with the Poincar\'e 
metric. The first author was partially supported by the project ANR 
NOSEVOL $2011$ BS $010119$ $01$.
\section{Main result}\label{mres}
Let $0< \delta \ll 1$ and consider the following random 
perturbation of $A_0$ as in \eqref{jb}: 
\begin{equation}\label{jb.p}
 A_{\delta} = A_0 + \delta Q, \quad Q=(q_{j,k})_{1 \leq j,k\leq N},
\end{equation}
where $q_{j,k}$ are independent and identically distributed 
complex random variables, following the complex Gaussian law 
$\mathcal{N}_{\C}(0,1)$.
\par
It has been observed by Bordeaux-Montrieux \cite{BM} the we have the following 
result.
\begin{prop}
There exists a $C_0>0$ such that the following holds: Let 
$X_j\sim \mathcal{N}_{\C}(0,\sigma_j^2)$, $1\leq j\leq N <\infty$ be 
independent and identically distributed complex Gaussian random 
variables. Put $s_1=\max \sigma_j^2$. Then, for every $x>0$, we 
have 
 \begin{equation*}
  \Proba\left[\sum_{j=1}^N |X_j|^2 \geq x\right]
    \leq \exp\left( 
    \frac{C_0}{2s_1} \sum_{j=1}^N \sigma_j^2 - \frac{x}{2s_1}
    \right).
 \end{equation*}
\end{prop}
According to this result we have 
$$
P(\Vert Q\Vert_\mathrm{HS}^2\ge x)\le \exp \left(\frac{C_0}{2}N^2-\frac{x}{2} \right)
$$
and hence if $C_1>0$ is large enough,
\begin{equation}\label{pj.28}
\Vert Q\Vert_\mathrm{HS}^2\le C_1^2N^2,\hbox{ with probability }\ge 1-e^{-N^2}.
\end{equation}
In particular (\ref{pj.28}) holds for the ordinary operator norm of $Q$. 
We now state the principal result of this work.
\begin{thm}\label{ed5}
Let $A_{\delta}$ be the $N\times N$-matrix in \eqref{jb.p} and restrict the
attention to the parameter range
$e^{-N/{\mO}(1)}\le \delta \ll 1$, $N\gg 1$. Let $r_0$ belong to a
parameter range,
$$
\frac{1}{{\mO}(1)}\le r_0\le 1-\frac{1}{N},
$$
\begin{equation}\label{ed.60}
\frac{r_0^{N-1}N}{\delta}(1-r_0)^2+\delta N^3\ll 1,
\end{equation}
so that $\delta \ll N^{-3}$. Then, for all $\varphi\in\mathcal{C}_0(D(0,r_0-1/N))$
  \begin{equation*}\label{ed.61}
   \E\left[ 
      \mathds{1}_{B_{\C^{N^2}}(0,C_1N)}(Q) 
      \sum_{\lambda\in\sigma(A_{\delta})}
      \varphi(\lambda)
      \right]
   =
    \frac{1}{2\pi}\int\varphi(z)\Xi(z) L(dz),
  \end{equation*}
 where 
 \begin{equation*}
  \Xi(z) =  \frac{4}{(1-|z|^2)^2}
    \left(
    1 + 
    \mO\!\left(\frac{|z|^{N-1}N}{\delta}(1-|z|)^2
	      +\delta N^3\right)
    \right).
 \end{equation*}
is a continuous function independent of $r_0$. 
$C_1>0$ is the constant in (\ref{pj.28}).
\end{thm}
Condition \eqref{ed.60} is equivalent to
\begin{equation*}
 r_0^{N-1}(1-r_0)^2 \ll \frac{\delta}{N}\left( 1 - \delta N^3\right).
\end{equation*}
It is necessary that $r_0 < 1 - 2(N+1)^{-1}$ for this inequality to 
be satisfied. For such $r_0$ the function $[0,r_0]\ni r \mapsto r^{N-1}(1-r)^2$ 
is increasing, and so inequality \eqref{ed.60} is preserved if we replace $r_0$ by 
$|z|\leq r_0$.
\\
\par
The leading contribution of the density $\Xi(z)$ is 
independent of $N$ and is equal to the Lebesgue density 
of the volume form induced by the Poincar\'e metric on the 
disc $D(0,1)$. This yields a very small density of 
eigenvalues close to the center of the disc $D(0,1)$ which is, 
however, growing towards the boundary of $D(0,1)$.
\\
\par
A similar result has been obtained by M. Sodin and B. Tsirelson 
in \cite{SoTs04} for the distribution of zeros of a certain class 
of random analytic functions with domain $D(0,1)$ linking the 
fact that the density is given by the volume form induced by 
the Poincar\'e metric on $D(0,1)$ to its invariance under 
the action of $SL_2(\R)$.
\subsection{Numerical Simulations}
To illustrate the result of Theorem \ref{ed5}, we present the 
following numerical calculations (Figure \ref{fig:e5_e4} and 
\ref{fig:e3_e2}) for the eigenvalues of the 
$N\times N$-matrix in \eqref{jb.p}, where $N=500$ and the coupling 
constant $\delta$ varies from $10^{-5}$ to $10^{-2}$. 
\begin{figure}[ht]
 \begin{minipage}[b]{0.49\linewidth}
  \centering
  \includegraphics[width=\textwidth]{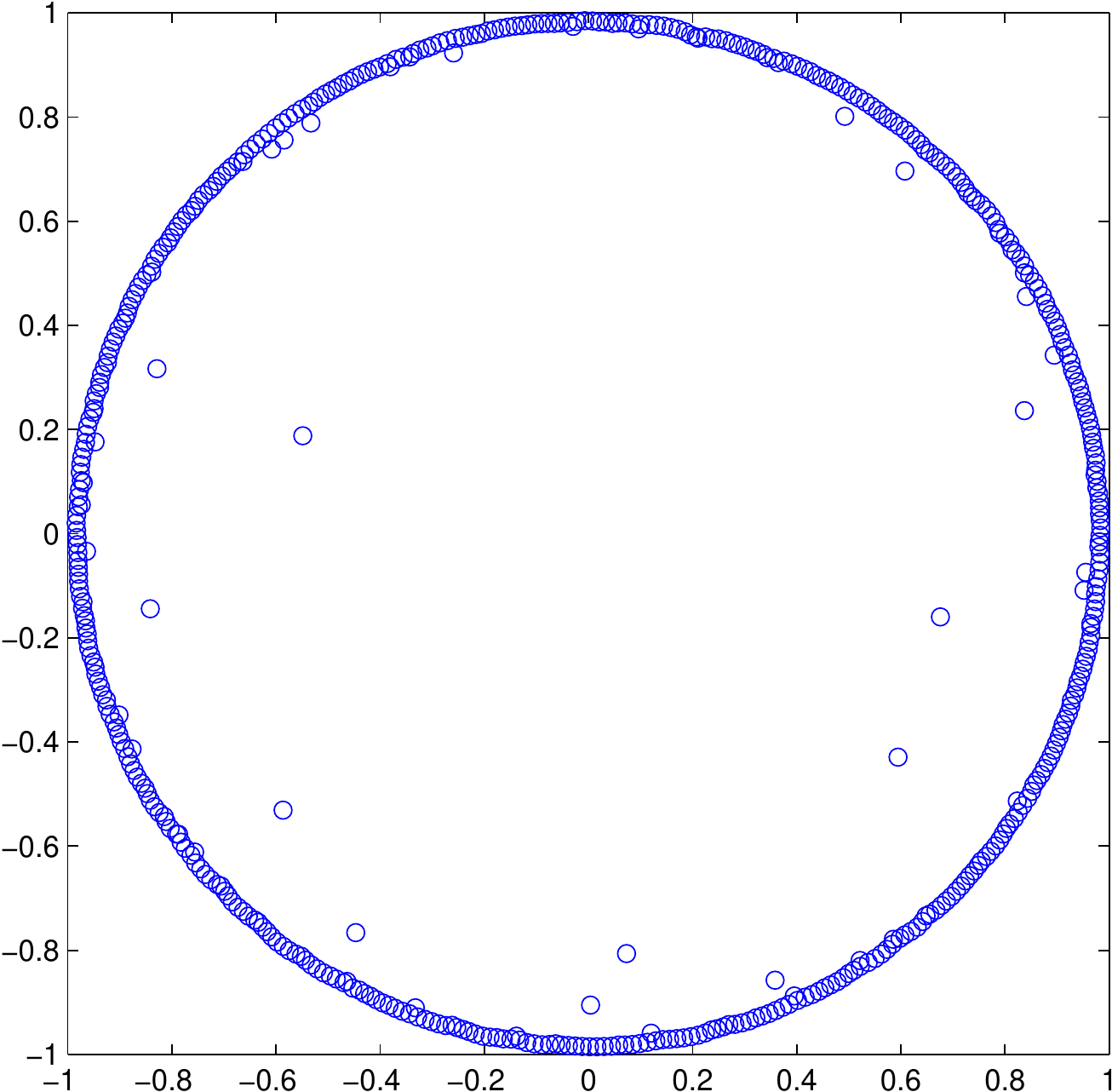}
 \end{minipage}
 \hspace{0cm}
 \begin{minipage}[b]{0.49\linewidth}
  \centering 
  \includegraphics[width=\textwidth]{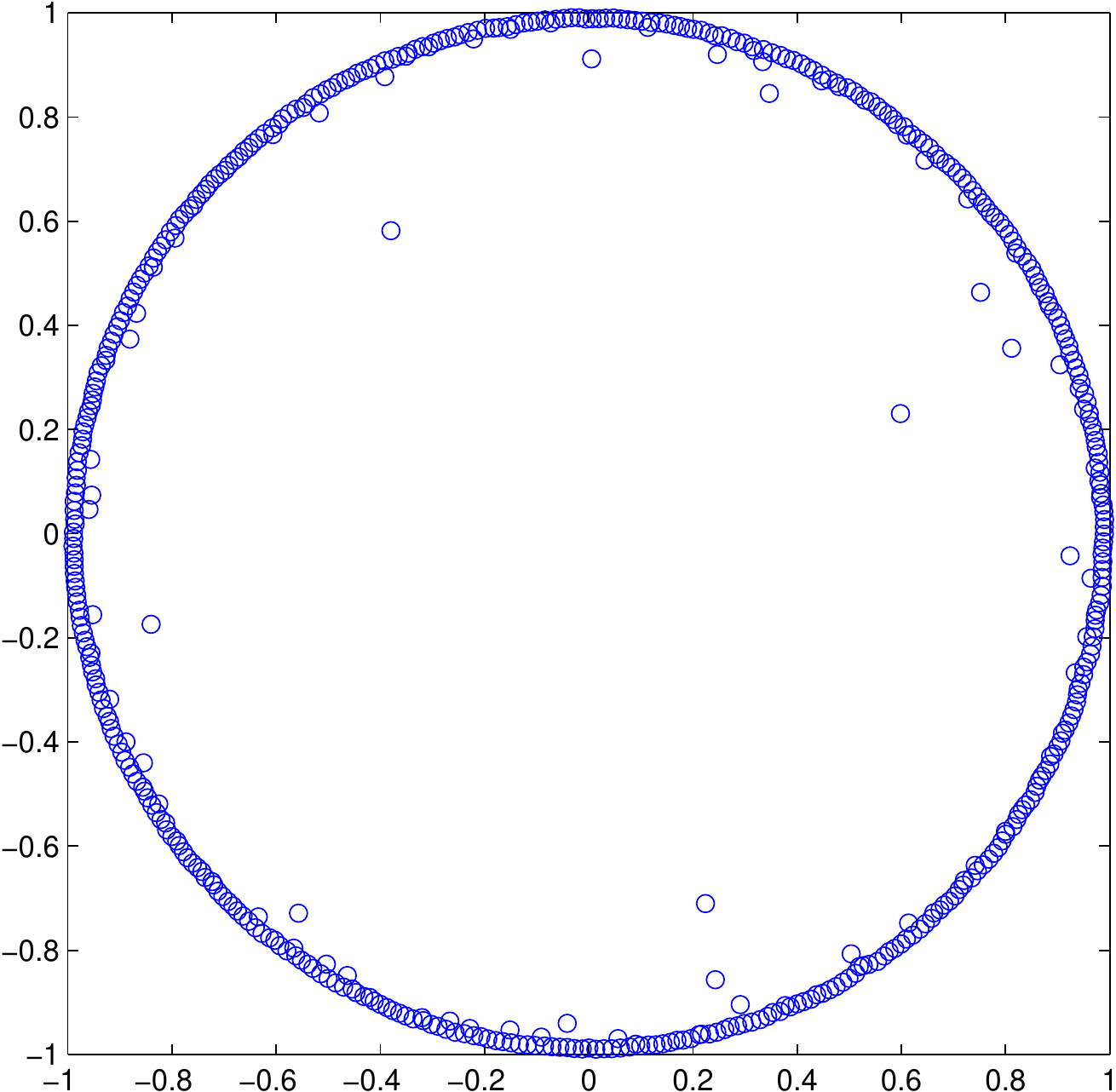}
 \end{minipage}
 \caption{On the left hand side $\delta=10^{-5}$ and 
  on the right hand side $\delta=10^{-4}$.}
  \label{fig:e5_e4}
\end{figure}
\begin{figure}[ht]
 \begin{minipage}[b]{0.44\linewidth}
  \centering
  \includegraphics[width=\textwidth]{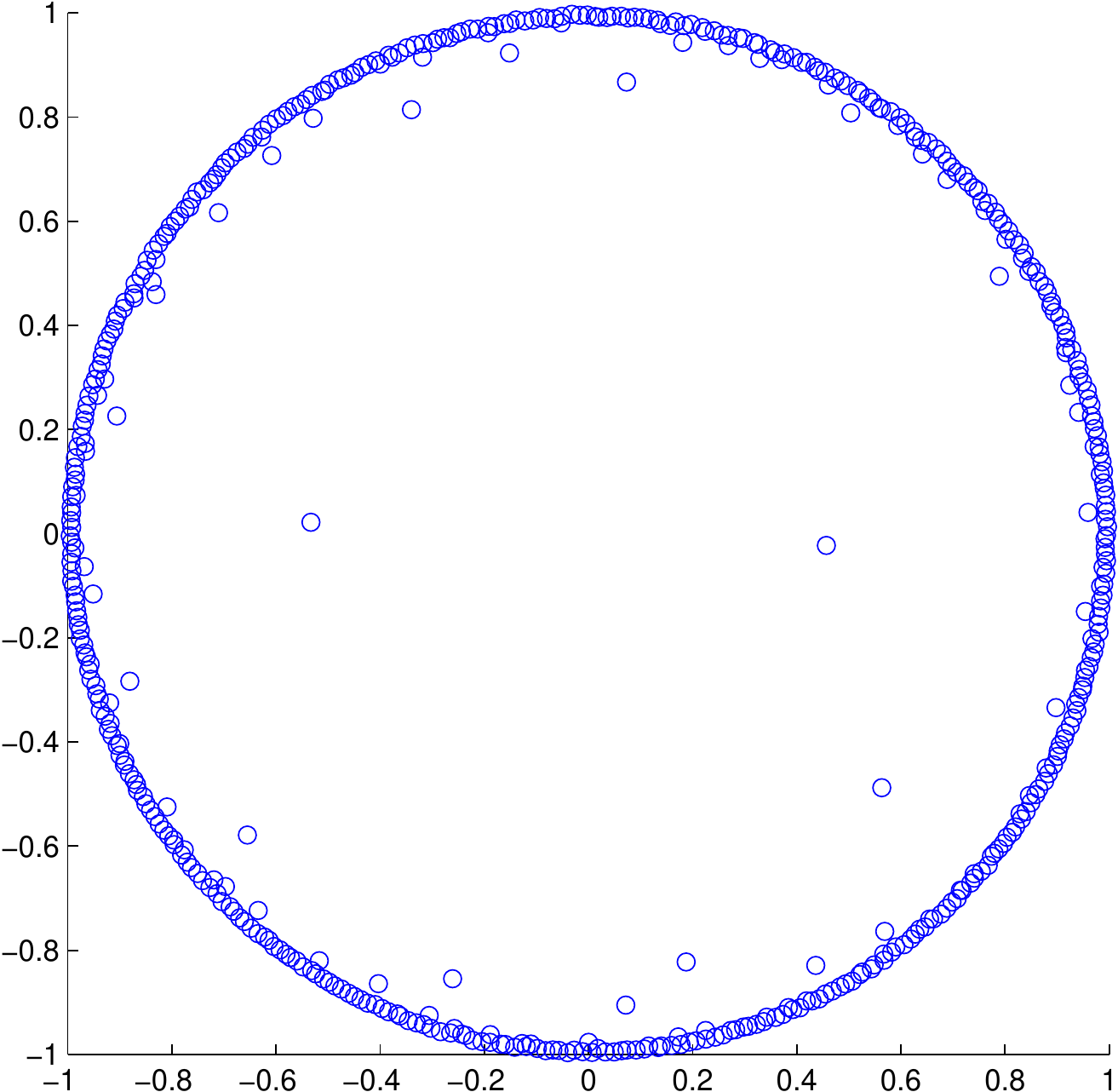}
 \end{minipage}
 \hspace{0cm}
 \begin{minipage}[b]{0.54\linewidth}
  \centering 
  \includegraphics[width=\textwidth]{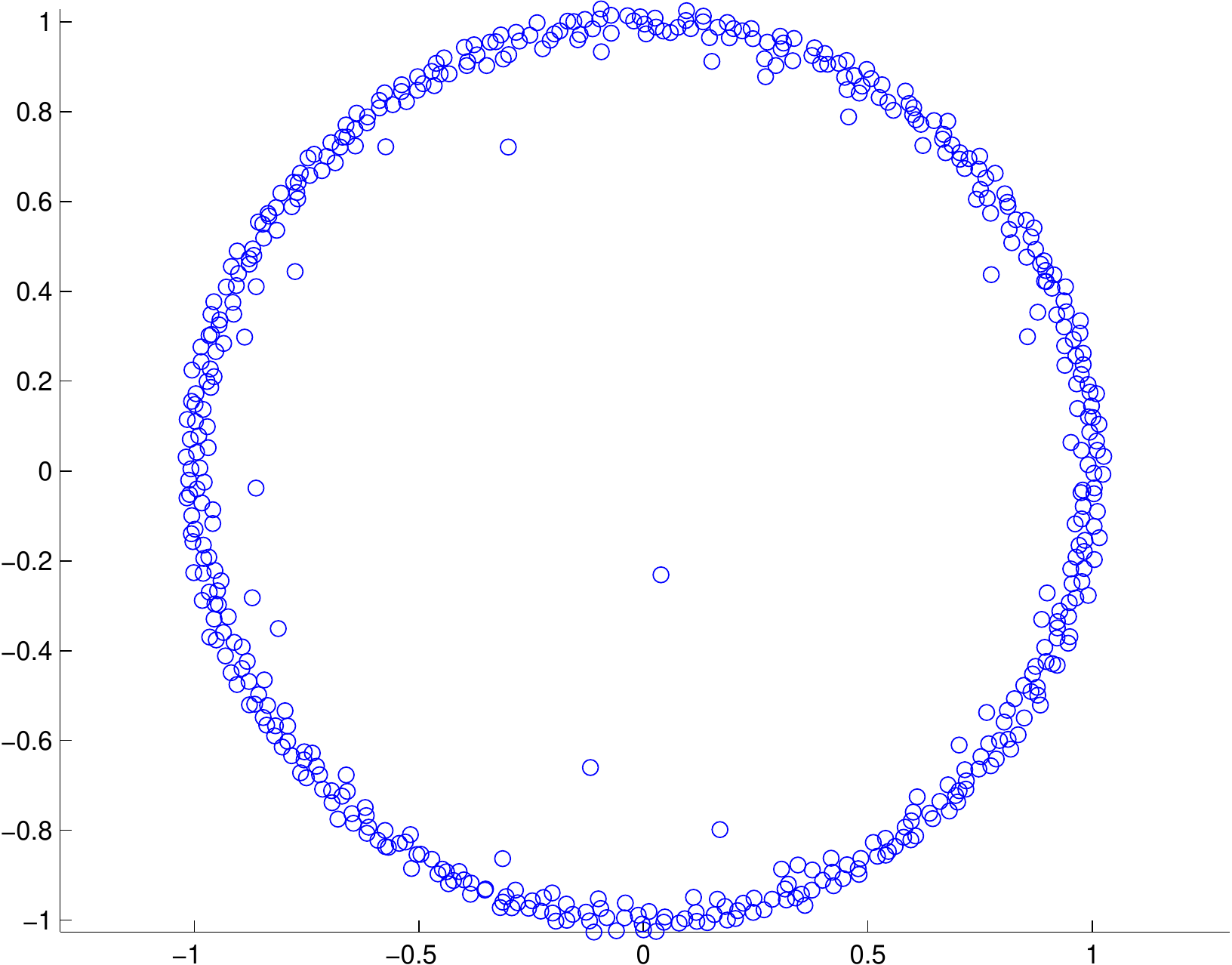}
 \end{minipage}
 \caption{On the left hand side $\delta=10^{-3}$ and 
  on the right hand side $\delta= 10^{-2}$.}
  \label{fig:e3_e2}
\end{figure}
\section{A general formula}
\setcounter{equation}{0}
To start with, we shall obtain a general formula (due to \cite{Vo14}
in a similar context). Our treatment is slightly different in that we
avoid the use of approximations of the delta function and also that we
have more holomorphy available.
\par
Let $g(z,Q)$ be a holomorphic function on $\Omega \times W\subset {\C}
\times {\C}^{N^2}$, where $\Omega\subset\C$, $W\subset {\C}^{N^2}$ 
are open bounded and connected. Assume that
\begin{equation}\label{ed.1} \hbox{for every }Q\in W,\
  g(\cdot ,Q)\not\equiv 0.
\end{equation}
To start with, we also assume that 
\begin{equation}\label{ed.2} 
\hbox{for almost all }Q\in W,\ g(\cdot ,Q) \hbox{ has only simple zeros.}
\end{equation}
Let $\phi \in C_0^\infty (\Omega )$ and let $m\in C_0(W)$. We are interested in 
\begin{equation}\label{ed.3}
K_\phi =\int \left(\sum_{z;\, g(z,Q)=0} \phi (z)\right) m(Q)L(dQ),
\end{equation}
where we frequently identify the Lebesgue measure with a differential
form,
\begin{equation*}
L(dQ)\simeq (2i)^{-{N^2}}d\overline{Q}_1\wedge dQ_1\wedge ... \wedge
d\overline{Q}_{N^2}\wedge dQ_{N^2} =:(2i)^{-{N^2}}d\overline{Q}\wedge dQ.
\end{equation*}
In (\ref{ed.3}) we
count the zeros of $g(\cdot ,Q)$ with their multiplicity and notice
that the integral is finite: For every compact set $K\subset W$ the
number of zeros of $g(\cdot ,Q)$ in $\mathrm{supp\,}\phi $, counted 
with their multiplicity, is uniformly bounded, for $Q\in K$. This 
follows from Jensen's formula.
\par 
Now assume,
\begin{equation}\label{ed.4}
g(z,Q)=0 \Rightarrow d_Qg\ne 0.
\end{equation}
Then 
$$
\Sigma :=\{(z,Q)\in \Omega \times W;\, g(z,Q)=0 \}
$$ 
is a smooth complex hypersurface in $\Omega \times W$ and from
(\ref{ed.2}) we see that 
\begin{equation}\label{ed.5}
K_\phi = \int_\Sigma \phi (z)m(Q)(2i)^{-{N^2}}d\overline{Q}\wedge dQ,
\end{equation}
where we view $(2i)^{-{N^2}}d\overline{Q}\wedge dQ$ as a complex $({N^2},{N^2})$-form on
$\Omega \times W$, restricted to $\Sigma $, which yields a non-negative differential form 
of maximal degree on $\Sigma$. 
\par 
Before continuing, let us eliminate the assumption
(\ref{ed.2}). Without that assumption, the integral in (\ref{ed.3}) is
still well-defined. It suffices to show \eqref{ed.5} for all 
$\phi\in\mathcal{C}_0^{\infty}(\Omega_0\times W_0)$ when $\Omega_0\times W_0$ 
is a sufficiently small open neighborhood of any given point 
$(z_0,Q_0)\in \Omega\times W$. When $g(z_0,Q_0)\neq 0$ or 
$\partial_z g(z_0,\Omega_0)\neq 0$ we already know that this holds, so we 
assume that for some $m\geq 2$, $\partial_z^kg(z_0,Q_0)=0$ for $0\leq k \leq m-1$, 
$\partial_z^mg(z_0,Q_0)\neq 0$. 
\par
Put $g_{\varepsilon}(z,Q)=g(z,Q)+\varepsilon$, $\varepsilon\in\mathrm{neigh}(0,\C)$. 
By Weierstrass' preparation theorem, if $\Omega_0,W_0$ and $r>0$ are small 
enough,
\begin{equation*}
 g_{\varepsilon}(z,Q) = k(z,Q,\varepsilon)p(z,Q,\varepsilon) \quad 
 \text{in } \Omega_0\times W_0\times D(0,r),
\end{equation*}
where $k$ is holomorphic and non-vanishing, and 
\begin{equation*}
 p(z,Q,\varepsilon)= z^m + p_1(Q,\varepsilon)z^{m-1} + \dots 
		     + p_m(Q,\varepsilon).
\end{equation*}
Here, $p_j(Q,\varepsilon)$ are holomorphic, and $p_j(0,0)=0$. 
\par
The discriminant $D(Q,\varepsilon)$ of the polynomial $p(\cdot,Q,\varepsilon)$ 
is holomorphic on $W_0\times D(0,r)$. It vanishes precisely when 
$p(\cdot,Q,\varepsilon)$ - or equivalently $g_{\varepsilon}(\cdot,Q)$ - has a 
multiple root in $\Omega_0$.
\par
Now for $0 < |\varepsilon| \ll 1$, the $m$ roots of $g_{\varepsilon}(\cdot,Q_0)$ 
are simple, so $D(Q_0,\varepsilon)\neq 0$. Thus, $D(\cdot,\varepsilon)$ is not 
identically zero, so the zero set of $D(\cdot,\varepsilon)$ in $W_0$ is of measure 
$0$ (assuming that we have chosen $W_0$ connected). This means that for 
$0<|\varepsilon|\ll 1$, the function $g_{\varepsilon}(\cdot,Q)$ has only simple 
roots in $\Omega$ for almost all $Q\in W_0$.
\par
Let $\Sigma _\epsilon $ be the zero set of
$g_\epsilon $, so that $\Sigma _\epsilon \to \Sigma $ in the natural
sense. We have 
$$
\int \left(\sum_{z;\, g_\epsilon (z,Q)=0 }\phi
(z)\right) m(Q)(2i)^{-{N^2}}d\overline{Q}\wedge dQ=\int_{\Sigma _\epsilon }\phi
(z)m(Q)(2i)^{-{N^2}}d\overline{Q}\wedge dQ
$$ 
for $\phi\in\mathcal{C}_0^{\infty}(\Omega_0\times W_0)$, 
when $\epsilon >0$ is small enough, depending on $\phi $, $m$. Passing
to the limit $\epsilon =0$ we get (\ref{ed.5}) under the assumptions
(\ref{ed.1}), (\ref{ed.4}), first for 
$\phi\in\mathcal{C}_0^{\infty}(\Omega_0\times W_0)$, and then by partition 
of unity for all $\phi\in\mathcal{C}_0^{\infty}(\Omega\times W)$. Notice that 
the result remains valid if we replace $m(Q)$ by $m(Q)1_B(Q)$ where $B$ is 
a ball in $W$.
\par
Now we strengthen the assumption (\ref{ed.4}) by assuming that we have
a non-zero $Z(z)\in {\C}^{N^2}$ depending smoothly on $z\in \Omega $
(the dependence will actually be holomorphic in the application below) 
such that
\begin{equation}\label{ed.6}
g(z,Q)=0 \Rightarrow \left(\overline{Z}(z)\cdot \partial _Q
\right)g(z,Q)\ne 0.
\end{equation}
We have the corresponding orthogonal decomposition
$$
Q=Q(\alpha )=\alpha _1\overline{Z}(z)+\alpha ',\quad \alpha '\in
\overline{Z}(z) ^\perp , \ \alpha_1\in\C,
$$ 
and if we identify unitarily $\overline{Z}(z)^\perp$ with ${\C}^{{N^2}-1}$ 
by means of an orthonormal basis $e_2(z),...,e_{N^2}(z)$, so
that $\alpha '=\sum_2^{N^2} \alpha _je_j(z)$ we get global coordinates
$\alpha _1,\alpha _2,...,\alpha _{N^2}$ on $Q$-space.
\par 
By the implicit function theorem, at least locally near any given
point in $\Sigma $, we can represent $\Sigma $ by $\alpha
_1=f(z,\alpha ')$, $\alpha '\in \overline{Z}(z)^\perp\simeq {\C}^{{N^2}-1}$, where $f$ is
smooth. (In the specific situation below, this will be valid
globally.) 
Clearly, since $z,\alpha _2,...,\alpha _{N^2}$ are complex coordinates on
$\Sigma $, we have on $\Sigma$ that
\begin{equation*}
  \frac{1}{(2i)^{N^2}}d\overline{Q}\wedge dQ
  =
  J(f)\frac{d\overline{z}\wedge dz}{2i} 
  (2i)^{1-N^2}d\overline{\alpha}_2\wedge d\alpha_2 \wedge ... 
  \wedge d\overline{\alpha}_{N^2}\wedge d\alpha_{N^2}
\end{equation*}
with the convention that 
\begin{equation*}
 J(f)\frac{d\overline{z}\wedge dz}{2i}  \geq 0, \quad
 (2i)^{1-N^2}d\overline{\alpha}_2\wedge d\alpha_2 \wedge ... 
\wedge d\overline{\alpha}_{N^2}\wedge d\alpha_{N^2} >0.
\end{equation*}
Thus 
\begin{equation}\label{ed.7}
\begin{split}
K_\phi =\int \phi (z)m\left(f(z,\alpha ')\overline{Z}(z)+\alpha
  '\right)J(f)(z,\alpha _2,...,\alpha _{N^2}) \times \\
(2i)^{-{N^2}}d\overline{z}\wedge dz\wedge d\overline{\alpha}_2\wedge d
\alpha_2 \wedge ... \wedge d\overline{\alpha}_{N^2}\wedge
d\alpha_{N^2} .
\end{split}
\end{equation}
The Jacobian $J(f)$ is invariant under any $z$-dependent unitary change
of variables, $\alpha _2,...,\alpha _{N^2}\mapsto \widetilde{\alpha
}_2,...,\widetilde{\alpha }_{N^2}$, so for the calculation of $J(f)$ at a
given point $(z_0,\alpha _0')$, we are free to choose the most
appropriate orthonormal basis $e_2(z),...,e_{N^2}(z)$ in
$\overline{Z}(z)^\perp$ depending smoothly on $z$. We write
(\ref{ed.7}) as 
\begin{equation}\label{ed.8}
K_\phi =\int \phi (z) \Xi (z) \frac{d\overline{z}\wedge dz}{2i},
\end{equation}
where the density $\Xi (z)$ is given by 
\begin{equation}\label{ed.9}\begin{split}
\Xi (z)=\int_{\alpha '=\sum_2^{N^2} \alpha _je_j(z)}m(f(z,\alpha
')\overline{Z}(z)+\alpha ')J(f)(z,\alpha _2,...,\alpha _{N^2})\times \\
(2i)^{1-{N^2}}d\overline{\alpha}_2\wedge d\alpha_2\wedge
... \wedge d\overline{\alpha}_{N^2}\wedge d\alpha_{N^2}.
\end{split}
\end{equation}
Before continuing, let us give a brief overview on the organization 
of following sections: 
\par
In Section \ref{gru} we will set up an auxiliary Grushin problem 
yielding the effective function $g$ as above. Section \ref{chvar} 
deals with the appropriate choice of coordinates $Q$ and the 
calculation of the Jacobian $J(f)$. Finally, in Section \ref{pfThm} 
we complete the proof of Theorem \ref{ed5}.

\section{Grushin problem for the perturbed Jordan block}\label{gru}
\subsection{Setting up an auxiliary problem}
Following \cite{SjZw07}, we introduce an auxiliary Grushin problem.
Define $R_+:{\C}^N\to {\C}$ by
\begin{equation}\label{pj.5}
R_+u=u_1,\ u=(u_1\ ...\ u_N)^\mathrm{t}\in {\C}^N. 
\end{equation} 
Let $R_-:\, {\C}\to {\C}^N$ be defined by
\begin{equation}\label{pj.6}
R_-u_-=(0\ 0\ ...\ u_-)^\mathrm{t}\in {\C}^N.
\end{equation}
Here, we identify vectors in ${\C}^N$ with column matrices. Then
for $|z|<1$, the operator
\begin{equation}\label{pj.7}
{\mathcal{A}}_0=\begin{pmatrix}A_0-z &R_-\\ R_+ &0\end{pmatrix}: 
{\C}^{N+1}\to {\C}^{N+1}
\end{equation}
is bijective. In fact, identifying 
\begin{equation*}
{\C}^{N+1}\simeq
\ell^2([1,2,...,N+1])\simeq\ell^2({\Z}/(N+1){\Z}), 
\end{equation*}
we have
${\mathcal{A}}_0=\tau ^{-1}-z\Pi _N$, where $\tau u(j)=u(j-1)$ (translation
by 1 step to the right) and $\Pi _Nu=1_{[1,N]}u$. Then ${\mathcal{A}}_0=\tau
^{-1}(1-z\tau \Pi _N)$, $(\tau \Pi _N)^{N+1}=0$,
$$
{\mathcal{A}}_0^{-1}=(1+z\tau \Pi _N+(z\tau \Pi _N)^2+...+(z\tau \Pi
_N)^N)\circ \tau .
$$
Write
$$
{\mathcal{E}}_0:={\mathcal{A}}_0^{-1}=:\begin{pmatrix}E^0 &E_+^0\\
E_-^0 &E_{-+}^0
\end{pmatrix}.
$$
Then 
\begin{equation}\label{pj.8}
E^0\simeq \Pi _N(1+z\tau \Pi _N+...(z\tau \Pi _N)^{N-1})\tau \Pi _N,
\end{equation}
\begin{equation}\label{pj.9}
E_+^0=\begin{pmatrix}1\\z\\ ..\\z^{N-1}\end{pmatrix},\
E_-^0=\begin{pmatrix}z^{N-1} & z^{N-2} &... &1\end{pmatrix},
\end{equation}
\begin{equation}\label{pj.10}
E_{-+}^0=z^N.
\end{equation}
A quick way to check \eqref{pj.9}, \eqref{pj.10} is to write 
$\mathcal{A}_0$ as an $(N+1)\times (N+1)$-matrix where we moved 
the last line to the top, with the lines labeled from 
$0$ ($\equiv N+1 \mod (N+1)\Z$) to $N$ and the columns from 
$1$ to $N+1$.
\par
Continuing, we see that
\begin{equation}\label{pj.10.2}
\Vert E^0\Vert\le G(|z|),\ \Vert E_\pm^0\Vert\le
G(|z|)^{\frac{1}{2}},\ \Vert E_{-+}^0\Vert\le 1 ,
\end{equation}
where $\Vert \cdot \Vert$ denote the natural operator norms and
\begin{equation}\label{pj.10.4}
G(|z|):=\min \left( N,\frac{1}{1-|z|} \right)\asymp
1+|z|+|z|^2+...+|z|^{N-1} .
\end{equation}
Next, consider the natural Grushin problem for $A_\delta $. If
$\delta \Vert Q\Vert G(|z|)<1$, we see that
\begin{equation}\label{pj.11}
{\mathcal{A}}_\delta =\begin{pmatrix}A_\delta -z  &R_-\\ R_+ &0\end{pmatrix}
\end{equation}
is bijective with inverse $${\mathcal{E}}_\delta =\begin{pmatrix}E^\delta
  & E_+^\delta \\ E_+^\delta  &E_{-+}^\delta \end{pmatrix},$$
where
\begin{equation}\label{pj.11.5}
\begin{split}
  E^\delta =&E^0-E^0\delta QE^0+E^0(\delta QE^0)^2-...=E^0(1+\delta
  QE^0)^{-1},\\
  E_+^\delta =&E_+^0-E^0\delta QE_+^0+(E^0\delta
  Q)^2E_+^0-...=(1+E^0\delta
  Q)^{-1}E^0_+,\\
  E_-^\delta =&E_-^0-E_-^0\delta QE^0+E_-^0(\delta
  QE^0)^2-...=E_-^0(1+\delta
  QE^0)^{-1},\\
  E^\delta _{-+}=&E^0_{-+}-E_-^0\delta QE_+^0+E_-^0\delta QE^0\delta
  QE_+^0-...\\
  &=E_{-+}^0-E_-^0\delta Q(1+E^0\delta Q)^{-1}E_+^0.
\end{split}
\end{equation}
We get
\begin{equation}\label{pj.12}
\begin{split}
&\Vert E^\delta \Vert\le \frac{G(|z|)}{1-\delta \Vert Q\Vert G(|z|)},\
\Vert E_\pm^\delta \Vert\le \frac{G(|z|)^{\frac{1}{2}}}{1-\delta \Vert
  Q\Vert G(|z|)},\\
&| E_{-+}^\delta - E_{-+}^0 |\le \frac{\delta \Vert Q\Vert G(|z|)}{1-\delta \Vert Q\Vert G(|z|)}.
\end{split}
\end{equation}
Indicating derivatives with respect to $\delta $ with dots and
omitting sometimes the super/sub-script $\delta $, we have
\begin{equation}\label{pj.13}
\dot{{\mathcal{E}}}=-{\mathcal{E}}\dot{{\mathcal{A}}}{\mathcal{E}}=-\begin{pmatrix}EQE &
  EQE_+\\ E_-QE &E_-QE_+.
\end{pmatrix}
\end{equation}
Integrating this from 0 to $\delta $ yields
\begin{equation}\label{pj.14}
\Vert  E^\delta -E^0\Vert \le \frac{G(|z|)^2 \delta \Vert Q \Vert}{(1-\delta \Vert Q\Vert
    G(|z|))^2} ,\quad 
    \Vert  E_\pm^\delta -E_\pm^0\Vert \le 
    \frac{G(|z|)^{\frac{3}{2}}\delta \Vert Q \Vert}{(1-\delta \Vert Q\Vert G(|z|))^2}.
\end{equation}
We now sharpen the assumption that $\delta \Vert Q\Vert G(|z|)<1$ to 
\begin{equation}\label{pj.15}
\delta \Vert Q\Vert G(|z|)<1/2. 
\end{equation}
Then 
\begin{equation}\label{pj.16}
\begin{split}
&\Vert E^\delta \Vert\le 2G(|z|),\
\Vert E_\pm^\delta \Vert\le 2 G(|z|)^{\frac{1}{2}},\\
&| E_{-+}^\delta - E_{-+}^0 |\le 2\delta \Vert Q\Vert G(|z|).
\end{split}
\end{equation}
Combining this with the identity $\dot{E}_{-+}=-E_-QE_+$ that follows
from (\ref{pj.13}), we get
\begin{equation}\label{pj.17}
\Vert \dot{E}_{-+}+E_-^0QE_+^0\Vert \le 16 G(|z|)^2\delta
  \Vert Q\Vert ^2,
\end{equation}
and after integration from $0$ to $\delta $,
\begin{equation}\label{pj.18}
E_{-+}^\delta =E_{-+}^0-\delta E_-^0QE_+^0 +{\mO}(1)G(|z|)^2(\delta
\Vert Q\Vert)^2 . 
\end{equation}
Using (\ref{pj.9}), (\ref{pj.10}) we get with $Q=(q_{j,k})$,
\begin{equation}\label{pj.19}
E_{-+}^\delta =z^N-\delta \sum_{j,k=1}^N q_{j,k}z^{N-j+k-1}
  +{\mO}(1)G(|z|)^2(\delta
\Vert Q\Vert)^2,
\end{equation}
still under the assumption (\ref{pj.15}).
\subsection{Estimates for the effective Hamiltonian}

\par We now consider the situation at the beginning of Section
\ref{mres}:
$$
A_\delta =A_0+\delta Q,\ Q=(q_{j,k}(\omega ))_{j,k=1}^N,\
q_{j,k}(\omega )\sim {\mathcal{N}}_{\C}(0,1)\hbox{ independent.}
$$
In the following, we often write $|\cdot|$ for the Hilbert-Schmidt norm 
$\lVert\cdot\rVert_{\mathrm{HS}}$. As we recalled in (\ref{pj.28}), we have 
\begin{equation}\label{ed.10}
|Q|\le C_1N\hbox{ with probability }\ge 1-e^{-N^2},
\end{equation}
and we shall work under the assumption that $|Q|\le C_1N$.
We let $|z|<1$ and assume:
\begin{equation}\label{ed.11}
\delta NG(|z|)\ll 1.
\end{equation}
Then with probability $\ge 1-e^{-N^2}$, we have (\ref{pj.15}),
(\ref{pj.19}) which give for $g(z,Q):=E_{-+}^\delta $, 
\begin{equation}\label{ed.12}
g(z,Q)=z^N+\delta (Q|\overline{Z}(z))+{\mO}(1)(G(|z|)\delta N)^2.
\end{equation}
Here, $Z$ is given by
\begin{equation}\label{ed.13}
Z=\left( z^{N-j+k-1}\right)_{j,k=1}^N.
\end{equation}
A straight forward calculation shows that
\begin{equation}\label{ed.14}
  | Z | =\sum_0^{N-1}|z|^{2\nu }=\frac{1-|z|^{2N}}{1-|z|^2}=
\frac{1-|z|^N}{1-|z|}\,\,\frac{1+|z|^N}{1+|z|},
\end{equation}
and in particular,
\begin{equation}\label{ed.15}
\frac{G(|z|)}{2}\le | Z| \le G(|z|).
\end{equation}
\par 
The middle term in (\ref{ed.12}) is bounded in modulus by $\delta
|Q||Z|\le \delta C_1NG(|z|)$ and we assume that $|z|^N$ is much
smaller than this bound:
\begin{equation}\label{ed.16}
|z|^N\ll \delta C_1NG(|z|).
\end{equation}
More precisely, we work in a disc $D(0,r_0)$, where 
\begin{equation}\label{ed.25}
r_0^N\le C^{-1}\delta C_1NG(r_0)\le C^{-2},\quad r_0\le 1-N^{-1}
\end{equation}
and $C\gg 1$. In fact, the first inequality in \eqref{ed.25} 
can be written $m(r_0) \leq C^{-1}\delta C_1 N$ and 
$m(r)= r^N(1-r)$ is increasing on $[0,1-N^{-1}]$ so the inequality 
is preserved if we replace $r_0$ by $|z|\leq r_0$. Similarly, the 
second inequality holds after the same replacement since $G$ is 
increasing. 
\par 
In view of (\ref{ed.11}), we see that
$$
\left(G(|z|)\delta N \right)^2\ll \delta G(|z|)N
$$
is also much smaller than the upper bound on the middle term. 

\par By the Cauchy inequalities,
\begin{equation}\label{ed.17}
d_Qg=\delta Z\cdot dQ+{\mO}(1)G(|z|)^2\delta ^2N.
\end{equation}
The norm of the first term is $\asymp \delta G\gg G^2\delta ^2N$,
since $G\delta N\ll 1$. (When applying the Cauchy inequalities, we
should shrink the radius $R=C_1N$ by a factor $\theta <1$, but we have
room for that, if we let $C_1$ be a little larger than necessary to
start with.)

Writing 
$$
Q=\alpha _1\overline{Z}(z)+\alpha ',\ \alpha '\in
\overline{Z}(z)^\perp\simeq {\C}^{{N^2}-1},
$$
we identify $g(z,Q)$ with a function $\widetilde{g}(z,\alpha )$ which is
holomorphic in $\alpha $ for every fixed $z$ and satisfies
\begin{equation}\label{ed.18}
\widetilde{g}(z,\alpha )=z^N+\delta |Z(z)|^2\alpha _1+{\mO}(1)
G(|z|)^2\delta ^2N^2,
\end{equation}
while (\ref{ed.17}) gives
\begin{equation}\label{ed.19}
\partial _{\alpha _1}\widetilde{g}(z,\alpha )=\delta |Z(z)|^2
+{\mO}(1)G(|z|)^3\delta ^2N,
\end{equation}
and in particular,
$$
\left| \partial _{\alpha _1}\widetilde{g} \right|\asymp \delta G(|z|)^2.
$$
This derivative does not depend on the choice of unitary
identification $\overline{Z}^\perp\simeq {\C}^{{N^2}-1}$. Notice that
the remainder in (\ref{ed.18}) is the same as in (\ref{ed.12}) and
hence a holomorphic function of $(z,Q)$. In particular it is a
holomorphic function of $\alpha _1,...,\alpha _{N^2}$ for every fixed $z$
and we can also get (\ref{ed.19}) from this and the Cauchy
inequalities. In the same way, we get from (\ref{ed.18}) that
\begin{equation}\label{ed.19.5}
\partial _{\alpha _j}\widetilde{g}(z,\alpha )={\mO}(1)G(|z|)^2\delta ^2N,\ j=2,...,{N^2}.
\end{equation}
\par 
The Cauchy inequalities applied to (\ref{ed.12}) give,
\begin{equation}\label{ed.26}
\partial _zg(z,Q)=Nz^{N-1}+\delta Q\cdot \partial _zZ(z)
+{\mO}(1)\frac{(G(|z|)\delta N)^2}{r_0-|z|}.
\end{equation}
Then, for $\widetilde{g}(z,\alpha _1,\alpha ')=g(z,\alpha
_1\overline{Z}(z)+\alpha ')$, $\alpha '= \sum_2^{N^2} \alpha _je_j$ 
we shall see that
\begin{equation}\label{ed.27}\begin{split}
\partial _z\widetilde{g}=
 Nz^{N-1}+\delta \alpha _1\partial _z\left( |Z|^2 \right)
 +{\mO}(1)\frac{(G\delta N)^2}{r_0-|z|}+{\mO}(1)G^2\delta
^2N\left| \sum_2^{N^2} \alpha _j\partial _ze_j \right|,
\end{split}
\end{equation}
\begin{equation}\label{ed.28}\begin{split}
\partial _{\overline{z}}\widetilde{g}
=& \delta \alpha _1\partial _{\overline{z}}\left( |Z|^2 \right) 
+{\mO}(1)G^2\delta ^2N\left| \alpha _1\overline{\partial
    _zZ}+\sum_2^{N^2} \alpha _j\partial _{\overline{z}}e_j \right| .
\end{split}\end{equation}
\par 
The leading terms in (\ref{ed.27}), (\ref{ed.28}) can be obtained
formally from (\ref{ed.18}) by applying $\partial _z$, $\partial
_{\overline{z}}$ and we also notice that 
$$
\partial _z|Z|^2=\overline{Z}\cdot \partial _zZ,\ \
\partial _{\overline{z}}|Z|^2=Z\cdot \overline{\partial _zZ }.
$$
However it is not clear how to handle the remainder in (\ref{ed.18}),
so we verify (\ref{ed.27}), (\ref{ed.28}), using (\ref{ed.17}),
(\ref{ed.26}):
\[
\begin{split}
&\partial _z\widetilde{g}=\partial _zg+d_Qg\cdot \sum_2^{N^2} \alpha
_j\partial _ze_j=\\
&Nz^{N-1}+\delta Q\cdot \partial _zZ+{\mO}(1)\frac{(G\delta
  N)^2}{r_0-|z|}+(\delta Z\cdot dQ +{\mO}(1)G^2\delta ^2N)\cdot
\sum_2^{N^2} \alpha _j\partial _ze_j\\
&= Nz^{N-1}+\delta \alpha _1\partial _z\left(|Z|^2 \right) +\delta
\sum_2^{N^2} \alpha _je_j\cdot \partial _zZ +\delta Z\cdot \sum_2^{N^2}\alpha
_j\partial _ze_j\\
&\hskip 5cm +\hbox{ the remainders in (\ref{ed.27})}.
\end{split}
\]
The 3d and the 4th terms in the last expression add up to 
$$
\delta \partial _z\left(\sum_2^{N^2} \alpha _je_j\cdot Z
\right)=\delta \partial _z(0)=0,
$$
and we get (\ref{ed.27}).

\par Similarly,
\[
\begin{split}
&\partial _{\overline{z}}\widetilde{g}=d_Qg\cdot \left(\alpha
  _1\overline{\partial _zZ}+\sum_2^{N^2} \alpha _j\partial
  _{\overline{z}}e_j \right)\\
&=\left(\delta Z\cdot dQ+{\mO}(1)G^2\delta ^2N  \right)\cdot \left(\alpha
_1\overline{\partial _zZ}+\sum_2^{N^2} \alpha _j\partial
_{\overline{z}}e_j\right) .
\end{split}
\]
Up to remainders as in (\ref{ed.28}), this is equal to 
\begin{align*}
\delta \alpha _1Z\cdot \overline{\partial _zZ}+\delta \sum_2^{N^2} \alpha
_jZ\cdot \partial _{\overline{z}}e_j &=
\delta \alpha _1\partial _{\overline{z}}\left(|Z|^2 \right) +\delta
\sum_2^{N^2} \alpha _j\partial _{\overline{z}}\left(Z\cdot e_j
\right)
 \notag \\
&=\delta \alpha _1 \partial _{\overline{z}}\left(|Z|^2 \right).
\end{align*}
\par 
Here, we know that 
$$
|Z(z)|=\sum_0^{N-1}(z\overline{z})^\nu =:K(z\overline{z}),
$$
\begin{equation}\label{ed.29}
\begin{split}
\partial _z\left(|Z(z)|^2 \right)&=2KK'\overline{z},\\
\partial _{\overline{z}}\left(|Z(z)|^2 \right)&=2KK'z.
\end{split}
\end{equation}
Observe also that $K(t)\asymp G(t)$ and that $G(|z|)\asymp G(|z|^2)$.

The following result implies that $K'(t)$ and $K(t)^2$ are of the same
order of magnitude.
\begin{prop}\label{ed1}
  For $k\in {\N}$, $2\le N\in {\N}\cup \{+\infty \}$, $0\le t<
  1$, we put
\begin{equation}\label{ed.29.1}
M_{N,k}(t)=\sum _{\nu =1}^{N-1} \nu ^kt^\nu ,
\end{equation}
so that $K(t)=K_N(t)=M_{N,0}(t)+1$, $K'(t)\asymp M_{N-1,1}(t)+1$. For each fixed $k\in {\N}$, we
have uniformly with respect to $N$, $t$:
\begin{equation}\label{ed.29.2}
M_{\infty ,k}(t)\asymp \frac{t}{(1-t)^{k+1}},
\end{equation}
\begin{equation}\label{ed.29.3}
M_{\infty ,k}(t)-M_{N,k}(t)\asymp \frac{t^N}{1-t}\left(N+\frac{1}{1-t} \right)^k.
\end{equation}
For all fixed $C>0$ and $k\in {\N}$, we have uniformly,
\begin{equation}\label{ed.29.4}
M_{N,k}(t)\asymp M_{\infty ,k}(t),\hbox{ for }0\le t\le
1-\frac{1}{CN},\ N\ge 2.
\end{equation}
\end{prop}

Notice that under the assumption in (\ref{ed.29.4}), the estimate
(\ref{ed.29.3}) becomes
$$
M_{\infty ,k}(t)-M_{N,k}(t)\asymp \frac{t^NN^k}{1-t}.
$$
We also see that in any region $1-{\mO}(1)/N\le t<1$, we have
$$
M_{N,k}(t)\asymp N^{k+1},
$$
so together with (\ref{ed.29.4}), \eqref{ed.29.2}, this shows that
\begin{equation}\label{ed.29.4.5}
M_{N,k}(t)\asymp t\min \left(\frac{1}{1-t},N \right)^{k+1}.
\end{equation}
\begin{proof}
The statements are easy to verify when $0\le t\le 1-1/{\mO}(1)$ and the
$N$-dependent statements (\ref{ed.29.3}), (\ref{ed.29.4}) are clearly
true when $N\le {\mO}(1)$. Thus we can assume that $1/2\le t<1$ and
$N\gg 1$. 

\par Write $t=e^{-s}$ so that $0<s\le 1/{\mO}(1)$ and notice
that $s\asymp 1-t$. For $N\in {\N}$, we put
\begin{equation}\label{ed.29.5}
P_{N,k}(s)=\sum_{\nu =N}^\infty \nu ^ke^{-\nu s},
\end{equation}
so that
\begin{equation}\label{ed.29.6}
P_{N,k}(s)=\begin{cases}M_{\infty ,k}(t)\hbox{ when }N=1,\\
M_{\infty ,k}(t)-M_{N,k}(t)\hbox{ when }N\ge 2.
\end{cases}
\end{equation}
We regroup the terms in (\ref{ed.29.5}) into sums with $\asymp 1/s$
terms where $e^{-\nu s}$ has constant order of magnitude:
$$
P_{N,k}(s)=\sum_{\mu =1}^{\infty }\Sigma (\mu ),\quad \Sigma (\mu
)=\sum_{N+\frac{\mu -1}{s}\le\nu < N+\frac{\mu }{s}}\nu ^ke^{-\nu s}.
$$
Here, since the sum $\Sigma (\mu )$ consists of $\asymp 1/s$ terms of the order
$\nu ^k e^{-(Ns+\mu )} $,
$$
\Sigma (\mu )\asymp e^{-(Ns+\mu )}\sum _{N+\frac{\mu -1}{s}\le\nu <
  N+\frac{\mu }{s}}\nu ^k \asymp e^{-(Ns+\mu )}\frac{(Ns+\mu )^k}{s^{k+1}}.
$$
Hence,
\[\begin{split}
  P_{N,k}(s)&\asymp \frac{e^{-Ns}}{s^{k+1}}\sum_{\mu =1}^\infty e^{-\mu
  }
  (Ns+\mu )^k\\
  &\asymp \frac{e^{-Ns}}{s^{k+1}}
  (Ns+1)^k=\frac{e^{-Ns}}{s}\left(N+\frac{1}{s}\right)^k.
\end{split}\]
Recalling (\ref{ed.29.6}) and the fact that $s\asymp 1-t$, $1/2\le
t<1$, we get
(\ref{ed.29.2}) when $N=1$ and (\ref{ed.29.3}) when $N\ge 2$. 

\par It remains to show (\ref{ed.29.4}) and it suffices to do so for
$1/2\le t\le 1-C/N$, $N\gg 1$ and for $C\ge 1$ sufficiently large but
independent of $N$. Indeed, for $1-C/N\le t\le 1-1/{\mO}(N)$, both
$M_{N,k}(t)$ and $M_{\infty ,k}(t)$ are $\asymp N^{1+k}$. We can also
exclude the case $k=0$ where we have explicit formulae. 

\par To get the equivalence (\ref{ed.29.4}) for $1/2\le t\le 1-C/N$,
$k\ge 1$,
it suffices, in view of (\ref{ed.29.2}), (\ref{ed.29.3}), to show that
for such $t$ and for $N\gg 1$, we have 
$$
\frac{N^kt^N}{1-t}\le \frac{1}{D}\frac{1}{(1-t)^{k+1}},
$$
for any given $D\ge 1$, provided that $C$ is large enough. In other
terms, we need
$$
t^N(1-t)^k\le \frac{1}{D}N^{-k},\hbox{ for }\frac{1}{2}\le t\le 1-\frac{C}{N},
$$
when $C=C(D)$ is large enough and $N\ge N(C)\gg 1$. The left hand side
in this inequality is an increasing function of $t$ on the interval
$[0,1/\left( 1+k/N \right)]$. If $t\le
1-C/N \le 1/(1+k/N)$ (which is fulfilled when $C\ge 2k$ and $N\gg N(C)$)
it is 
$$
\le \left(1-\frac{C}{N} \right)^N\left(\frac{C}{N} \right)^k=\left( 1
+{\mO}_C\left(\frac{1}{N} \right)\right)e^{-C}C^kN^{-k}.
$$
This is $\le N^{-k}/D$ if $C\ge C(D)$, $N\ge N(C)$.
\end{proof}
For simplicity we will restrict the attention to the region
\begin{equation}\label{ed.30}
|z|\le r_0-1/N,
\end{equation}
where $G\asymp (1-|z|)^{-1}$, $G'\asymp (1-|z|)^{-2}$.
\par 
It follows from the calculation (\ref{ed.35.5}) below, that 
$$
|\partial _zZ|^2=\left(\frac{2}{t}\left(K(t\partial _t)^2K
		  +(t\partial _tK)^2 \right) \right)_{t=|z|^2}.
$$
This is $\asymp 1$ for $|z|\le 1/2$ and for $1/2\le |z|<1-1/N $ it is in 
view of Proposition \ref{ed1} and the subsequent observation
$$\asymp M_{N,0}M_{N,2}+M_{N,1}^2\asymp \frac{1}{(1-t)^4},\ t=|z|^2.$$ 
In the region (\ref{ed.30}) we get:
\begin{equation}\label{ed.31}
|Z'(z)|\asymp G(|z|)^2.
\end{equation}
(\ref{ed.29}), (\ref{ed.30}), (\ref{ed.31}) will be used in
(\ref{ed.27}), (\ref{ed.28}). 
\par 
Combining the implicit function theorem and Rouch\'e's theorem to 
\eqref{ed.18},we see that for 
$|\alpha '|<C_1N$, $\alpha '=\sum_2^N \alpha _je_j\in 
\overline{Z}(z)^\perp$, the equation
\begin{equation}\label{ed.20}
\widetilde{g}(z,\alpha _1,\alpha ')=0
\end{equation}
has a unique solution
\begin{equation}\label{ed.21}
\alpha _1=f(z,\alpha ')\in D(0,C_1N/G(|z|)).
\end{equation}
Here, we also use \eqref{ed.11}, \eqref{ed.16}. Moreover, $f$ satisfies
\begin{equation}\label{ed.22}f(z,\alpha ')=-\frac{z^N}{\delta
    |Z|^2}+{\mO}(1)\delta N^2={\mO}(1)\left(\frac{|z|^N}{\delta
      G^2}+\delta N^2 \right).\end{equation}

\par Differentiating the equation (\ref{ed.20}) (where $\alpha _1=f$)
we get
$$
\partial _z\widetilde{g}+\partial _\alpha \widetilde{g}\partial
_zf=0,\ \partial _{\overline{z}}\widetilde{g}+\partial _\alpha \widetilde{g}\partial
_{\overline{z}}f=0.
$$
Hence,
\begin{equation}\label{ed.23}
\begin{cases}
\partial _zf=-\left(\partial _{\alpha _1}\widetilde{g}
\right)^{-1}\partial _z\widetilde{g},\\
\partial _{\overline{z}}f=-\left(\partial _{\alpha _1}\widetilde{g}
\right)^{-1}\partial _{\overline{z}}\widetilde{g}.
\end{cases}
\end{equation}
Since $\widetilde{g}$ is holomorphic in $\alpha _1,\alpha '$ and in
$\alpha _1,\alpha _2,...,\alpha _{N^2}$, we see that $f$ is holomorphic in
$\alpha '$ and in $\alpha _2,...,\alpha _{N^2}$ Applying $\partial
_{\alpha _2},...,\partial _{\alpha _{N^2}}$ to (\ref{ed.20}), we get
\begin{equation}\label{ed.24}
\partial _{\alpha _j}f=-\left(\partial _{\alpha _1}\widetilde{g}
\right)^{-1} \partial _{\alpha _j}\widetilde{g},\ 2\le j\le {N^2}.
\end{equation}
\par 
Combining (\ref{ed.19}) in the form,
$$
\partial _{\alpha _1}\widetilde{g}(z,\alpha )=(1+{\mO}(G(|z|)\delta
N))\delta |Z(z)|^2,
$$
(\ref{ed.19.5}), (\ref{ed.27}), (\ref{ed.28}) with (\ref{ed.23}) and
(\ref{ed.24}), we get
\begin{equation}\label{ed'.1}
\begin{split}
&\partial _zf=-\frac{(1+{\mO}(G\delta N))}{\delta |Z(z)|^2}\times
\\ & \left(Nz^{N-1}+\delta
f\partial _z\left(|Z|^2 \right)
+{\mO}\left( G^2\delta ^2N\right)\left|\sum_2^{N^2} \alpha _j\partial _ze_j
\right|
+{\mO}(1)\frac{(G\delta N)^2}{r_0-|z|}
\right).
\end{split}
\end{equation}
\begin{equation}\label{ed'.2}
\begin{split}
\partial _{\overline{z}}f=&-\frac{(1+{\mO}(G\delta N))}{\delta
  |Z(z)|^2}\times \\
&\left(\delta f\partial _{\overline{z}}\left(|Z|^2 \right)
+{\mO}\left( G^2\delta ^2N\right)\left| f\overline{\partial _zZ}
  +\sum_2^{N^2} \alpha _j\partial _{\overline{z}}e_j \right|
 \right),
\end{split}
\end{equation}
\begin{equation}\label{ed'.3}
\partial _{\alpha _j}f={\mO}(1)\frac{G^2\delta ^2N}{\delta
  G^2}={\mO}(\delta N),\ 2\le j\le {N^2}.
\end{equation}
From (\ref{ed.29}) and the observation prior to Proposition \ref{ed1}
we know that 
$$
\partial _z\left( |Z|^2 \right),\ \partial _{\overline{z}}\left( |Z|^2
\right) \asymp G(|z|)^3|z|.
$$
Recall also that $|Z|\asymp G(|z|)$. Using this in (\ref{ed'.1}),
(\ref{ed'.2}), we get
\begin{equation}\label{ed'.4}
\begin{split}
\partial _zf=&\frac{{\mO}(1)}{\delta G^2}\times \\
&\left(N|z|^{N-1}+\delta |f|G^3|z|
+{\mO}\left(G^2\delta ^2N \right)\left| \sum_2^{N^2} \alpha _j\partial
    _ze_j \right| +{\mO}(1)\frac{G^2\delta ^2N^2}{r_0-|z|})
\right) .
\end{split}
\end{equation}
\section{Choosing appropriate coordinates}\label{chvar}
\par 
The next task will be to choose an orthonormal basis
$e_1(z),e_2,...,e_{N^2}(z)$ in ${\C}^{N^2}$ with
$e_1(z)=|Z(z)|^{-1}\overline{Z}(z)$ such that we get a nice control
over $\sum_2^{N^2}\alpha _j\partial _ze_j$, $\sum_2^{N^2}\alpha _j\partial
_{\overline{z}}e_j$ and such that 
$$
{{dQ_1\wedge ...\wedge dQ_{N^2}}_\vert}_{\alpha _1=f(z,\alpha ')}
$$
can be expressed easily up to small errors. Consider a point $z_0\in
D(0,r_0-N^{-1})$. We shall see
below that the vectors $\overline{Z}(z)$, $\overline{\partial _zZ}(z)$
are linearly independent for every $z\in D(0,1)$
\begin{prop}\label{ed2}
There exists an orthonormal basis $e_1(z),e_2(z),...,e_{N^2}(z)$ in ${\C}^{N^2}$, 
depending smoothly on $z\in \mathrm{neigh\,}(z_0)$ such that
\begin{equation}\label{ed'.6}
e_1(z)=|Z(z)|^{-1}\overline{Z}(z),
\end{equation}
\begin{equation}\label{ed'.7}
{\C}e_1(z_0)\oplus{\C}e_2(z_0)={\C}\overline{Z}(z_0)\oplus
\overline{\partial _zZ}(z_0),
\end{equation}
\begin{equation}\label{ed'.8}
e_j(z)-e_j(z_0)={\mO}((z-z_0)^2),\ j\ge 3.
\end{equation}
\end{prop}
\begin{proof}
We choose $e_1(z)$ as in (\ref{ed'.6}). Let $e_3(z_0),...,e_{N^2}(z_0)$ be
an orthonormal basis in $\left({\C}\overline{Z}(z_0)\oplus {\C}
\overline{\partial _zZ}(z_0) \right)^\perp$. Then we get an
orthonormal family $e_3(z),...,e_{N^2}(z)$ in $e_1(z)^\perp$ in the
following way:

\par Let $V_0$ be the isometry ${\C}^{{N^2}-2}\to {\C}^{N^2}$, defined
by $V_0\nu _j^0=e_j(z_0)$, $j=3,...,{N^2}$, where $\nu _3^0,...,\nu _{N^2}^0$
is the canonical basis in ${\C}^{{N^2}-2}$ with a non-canonical
labeling. Let $\pi (z)u=(u|e_1(z))e_1(z)$ be the orthogonal
projection onto ${\C}e_1(z)$. For $z\in \mathrm{neigh\,}(z_0,{\C})$, 
let $V(z)=(1-\pi (z))V_0$. Then $f_j(z)=V(z)\nu _j^0$,
$j=3,...,{{N^2}}$ form a linearly independent system in $e_1(z)^\perp$ and
we get an orthonormal system of vectors that span the same hyperplane
in $e_1(z)^\perp$ by Gram orthonormalization,
$$
e_j(z)=V(z)(V^*(z)V(z))^{-\frac{1}{2}}\nu _j^0,\ 3\le j\le {N^2}.
$$

\par We have 
$$
V(z)\nu _j^0=(1-\pi (z))e_j(z_0)=e_j(z_0)-(e_j(z_0)|e_1(z))e_1(z),
$$
$$
(e_j(z_0)|e_1(z))=\frac{(e_j(z_0)|\overline{Z}(z))}{|Z(z)|}={\mO}((z-z_0)^2),
$$
since $(e_j(z_0)|\overline{Z}(z))=e_j(z_0)\cdot Z(z)=:k(z)$ is a
holomorphic function of $z$ with
$k(z_0)=(e_j(z_0)|\overline{Z}(z_0))=0$,
$k'(z_0)=(e_j(z_0)|\overline{\partial _zZ}(z_0))=0$.
Thus, $V(z)=V(z_0)+{\mO}(z-z_0)^2)$ and we conclude that
(\ref{ed'.8}) holds. Let $e_2(z)$ be a normalized vector in
$(e_1(z),e_3(z),e_4(z),...,e_{N^2}(z))^\perp$ depending smoothly on
$z$. Then $e_1(z),e_2(z),...,e_{N^2}(z)$ is an orthonormal basis and since
$e_3(z_0),...,e_{N^2}(z_0)$ are orthogonal to
$\overline{Z}(z_0),\overline{\partial Z}(z_0)$ by construction, we get (\ref{ed'.7}).
\end{proof}
\par 
We can make the following explicit choice:
\begin{equation}\label{ed'.9}
e_2(z)=|f_2|^{-1}f_2,\ f_2=\overline{\partial _zZ}(z)-\sum_{j\ne
  2}(\overline{\partial _zZ}(z)|e_j(z))e_j(z),
\end{equation}
so that for $z=z_0$,
\begin{equation}\label{ed'.10}
e_2(z_0)=|f_2(z_0)|^{-1}f_2(z_0),\ f_2(z_0)= 
\overline{\partial _zZ}(z_0)-(\overline{\partial _zZ}(z_0)|e_1(z_0))e_1(z_0).
\end{equation}
\par 
We next compute some scalar products and norms with $Z$ and
$\partial _zZ$. Recall that $Z(z)=\left(z^{N-j+k-1} \right)_{j,k=1}^N$
and that $|Z(z)|=K(|z|^2)$, $K(t)=\sum_0^{N-1}t^\nu $. Repeating
basically the same computation, we get
$$
z\partial _zZ=\left((N-j+k-1)z^{N-j+k-1} \right)_{j,k=1}^N,
$$
and
\begin{equation}\label{ed.35.5}\begin{split}
&|z\partial _zZ|^2=\sum_{j,k=1}^N (N-j+k-1)^2|z|^{2(N-j+k-1)}
=\sum_{\nu ,\mu =0 }^{N-1}(\nu +\mu )^2|z|^{2(\nu +\mu )}\\
&=\sum_0^{N-1}\nu ^2|z|^{2\nu }\sum_0^{N-1}|z|^{2\mu }+2\sum_0^{N-1}\nu
|z|^{2\nu }\sum_0^{N-1}\mu |z|^{2\mu }+\sum_0^{N-1}|z|^{2\nu
}\sum_0^{N-1}\mu ^2|z|^{2\mu }\\
&=2\left(K(t\partial _t)^2K + (t\partial _tK)^2 \right)_{t=|z|^2}.
\end{split}
\end{equation}
Similarly,
\[
\begin{split}
(z\partial _zZ|Z)=&\sum_{j,k=1}^N (N-j+k-1)|z|^{2(N-j+k-1)}\\
=&\sum_{\nu =0}^{N-1}\sum_{\mu =1}^{N-1} (\nu +\mu )|z|^{2(\nu +\mu
  )}\\
=&2(Kt\partial _tK)_{t=|z|^2}.
\end{split}
\]
Then, by a straight forward calculation,
\begin{equation}\label{ed.36}
|\partial _zZ|^2-\frac{|(\partial
  _zZ|Z)|^2}{|Z|^2}=\left(\frac{2}{t}\left(K(t\partial _t)^2K -(t\partial _tK)^2 \right)\right)_{t=|z|^2}
\end{equation}
\par 
Here,
\[\begin{split}
\frac{2}{t}\left(K(t\partial _t)^2K-(t\partial _tK)^2 \right)
&=\frac{2}{t}\sum_0^{N-1}t^\nu \sum_0^{N-1}\nu ^2t^\nu
-\frac{2}{t}\left(\sum_0^{N-1} \nu t^\nu  \right)^2\\
=\sum_{\nu ,\mu =0}^{N-1} \left(\nu ^2+\mu ^2-2\nu \mu
  \right)t^{\nu +\mu -1}  
&=\sum_{\nu ,\mu =0}^{N-1}(\nu -\mu )^2t^{\nu +\mu -1}\\
&=\sum_{k=0}^{2N-3}a_{k,N}t^k, 
\end{split}
\]
where
$$
a_{k,N}=\sum_{\substack{\nu +\mu -1=k \\ 0\le \nu ,\mu \le N-1}}(\nu -\mu )^2.
$$
We observe that
$$
a_{k,N}\le {\mO}(1)(1+k)^3\hbox{ uniformly with respect to }N,
$$
$$a_{k,N}=a_{k,\infty }\hbox{ is independent of }N\hbox{ for }k\le N-2,$$
$$
a_{k,\infty }\ge (1+k)^3/{\mO}(1).
$$
We conclude that
$$
\frac{1}{C}\left(1+M_{N-1,3} \right)\le
\frac{2}{t}\left(K(t\partial _t)^2K-(t\partial _tK)^2 \right)
 \le C\left(1+M_{2N-2,3} \right)
$$
and \eqref{ed.29.4.5} shows that the first and third members are of the same
order of magnitude,
$$
\asymp 1+M_{N,3}(t)\asymp \min \left(\frac{1}{1-t},N \right)^4
$$
which is $\asymp 1+ M_{\infty ,3}(t)$, for $0\le t\le 1-1/N$. From this
and Proposition \ref{ed1} we get:
\begin{prop}\label{ed3}
We have
\begin{equation}\label{ed.37}
\frac{2}{t}(K(t\partial _t)^2K-(t\partial _tK)^2)\asymp K^4,\ 0<t\le 1-1/N,
\end{equation}
where we recall that $K=K_N$ depends on $N$ and that 
$$
K_N=K_\infty -\frac{t^N}{1-t}.
$$
We have
\begin{equation}\label{ed.38}
\begin{cases}
t\partial _tK_N=t\partial _tK_\infty +{\mO}\left(\frac{Nt^N}{1-t}
\right),\ &t\le 1-\frac{1}{N},\\
(t\partial _t)^2K_N=(t\partial _t)^2K_\infty +{\mO}\left(\frac{N^2t^N}{1-t}
\right),\ &t\le 1-\frac{1}{N},
\end{cases}
\end{equation}
and it follows that
\begin{align}\label{ed.39}
  &\frac{2}{t}\left( K_N(t\partial _t) ^2K_N-(t\partial _tK_N) ^2\right)-
  \frac{2}{t}\left( K_\infty (t\partial _t) ^2K_\infty -(t\partial _t
    K_\infty) ^2 \right)
    \notag \\
  &= {\mO}\left(\frac{N^2t^N}{(1-t)^2} \right),
\end{align}
for $t\le 1-1/N$.
\end{prop}
Proposition \ref{ed3} and (\ref{ed.36}) give
\begin{equation}\label{ed.40}
|\partial _zZ|^2-\frac{|(\partial _zZ|Z)|^2}{|Z|^2}\asymp K(|z|^2)^4.
\end{equation}
This implies that $\partial _zZ$, $Z$ are linearly independent.
\par Assume that $$|\nabla _ze_1(z)|={\mO}(m)$$
for some weight $m\ge 1$. We shall see below that this holds when
$m=K(|z|^2)$. Then $\| \nabla _z\Pi \|={\mO}(m)$ and hence
$\|\nabla_zV \|={\mO}(m)$. It follows that $\| \nabla
_z(V^*(z)V(z))\|={\mO}(m)$. By standard (Cauchy-Riesz)
functional calculus, using also that $\| V(z)^{-1}\|={\mO}(1)$, we
get $\| \nabla _z(V^*(z)V(z))^{-\frac{1}{2}}\| = {\mO}(m)$. Hence
$\| \nabla _zU(z)\|={\mO}(m)$, where $U(z)=V(z)(V^*(z)V(z))^{-1/2}$
is the isometry appearing in the proof of Proposition \ref{ed2}. Since $\nabla _ze_j=(\nabla
_zU(z))\nu _j^0$, we conclude that $\| \nabla _zU(z)\| ={\mO}(m)$, so 
\begin{equation}\label{ed.41}
\left| \sum_3^{N^2} \alpha _j\nabla _ze_j \right| \le {\mO}(m)\|\alpha
\|_{{\C}^{{N^2}-2}}.
\end{equation}
We next show that we can take $m=K(|z|^2)$. We have
\begin{equation}\label{ed.42}
\nabla _ze_1=\frac{\nabla _z\overline{Z}}{|Z|}-\frac{\nabla
  _z|Z|}{|Z|^2}\overline{Z}=\frac{\nabla
  _z \overline{Z}}{K}-\frac{K'\nabla _z(z\overline{z})}{K^2}\overline{Z}.
\end{equation}
By \eqref{ed.35.5},
$$
|\partial _z Z|=\left(\frac{2}{t}\left(K(t\partial
    _t)^2K+(t\partial _tK)^2 \right)
\right)^{\frac{1}{2}}_{t=|z|^2}={\mO}(K^2).
$$
Since $Z$ is holomorphic, this leads to the same estimates for
$|\nabla _zZ|$ and $|\nabla _z\overline{Z}|$, and 
$|\partial_z^2 Z| = \mO(K^3)$, for $|z|< 1-N^{-1}$, by the Cauchy 
inequalities. Using this in
(\ref{ed.42}), we get
\begin{equation}\label{ed.43}
|\nabla _ze_1|={\mO}(K).
\end{equation}
Thus we can take $m=K(|z|^2)$ in (\ref{ed.41}). Let $f_2$ be the
vector in (\ref{ed'.9}) so that $e_2(z)=|f_2|^{-1}f_2$. Recall that
$e_j=U(z)\nu _j^0$, where we now know that $\| \nabla _zU(z)\| ={\mO}(K)$. 
Write,
$$
\nabla _zf_2=\nabla _z\overline{\partial _zZ}-\sum_{j\ne 2}\left(
(\nabla _z\overline{\partial _zZ}|e_j)e_j+(\overline{\partial
  _zZ}|\nabla _ze_j)e_j+(\overline{\partial _zZ}|e_j)\nabla _ze_j
 \right).
$$

\par Here, $|\nabla _z\overline{\partial _zZ}|={\mO}(K^3)$, as we
have just seen. It is also clear that the term for $j=1$ in the sum
above is ${\mO}(K^3)$. It remains to study
$|\mathrm{I}+\mathrm{II}+\mathrm{III}|\le
|\mathrm{I}|+|\mathrm{II}|+|\mathrm{III}|$, where
\[\begin{split}
\mathrm{I}&=\sum_3^{N^2}(\nabla _z\overline{\partial _zZ}|e_j)e_j,\\
\mathrm{II}&=\sum_3^{N^2}(\overline{\partial
  _zZ}|\nabla _ze_j)e_j,\\
\mathrm{III}&=\sum_3^{N^2}(\overline{\partial _zZ}|e_j)\nabla _ze_j.
\end{split}
\]

\par Here, $|\mathrm{I}|\le |\nabla _z\overline{\partial _zZ}|=
{\mO}(K^3)$ and by (\ref{ed.41}) we have 
$|\mathrm{III}|\le {\mO}(K)|\overline{\partial _zZ}|={\mO}(K^3)$. Further,
\[
\begin{split}
\mathrm{II}&=\sum_3^{N^2} (\overline{\partial _zZ}|(\nabla _zU(z))\nu
_j^0)e_j\\
&=\sum _3^{N^2}((\nabla _zU(z))^*\overline{\partial _zZ}|\nu _j^0)e_j,
\end{split}
\]
so
$$
|\mathrm{II}|=|(\nabla _zU(z))^*\overline{\partial _zZ}|={\mO}(K)K^2={\mO}(K^3).
$$
Thus, 
\begin{equation}\label{ed.43.5}
|\nabla _zf_2|={\mO}(K^3).
\end{equation}
Recall from (\ref{ed'.10}) that for $z=z_0$,
$$
f_2=\overline{\partial _zZ}-(\overline{\partial _zZ}|e_1)e_1,
$$
$$
|f_2|^2=|\partial _zZ|^2-\frac{|(\partial _zZ|Z)|^2}{|Z|},
$$
so by (\ref{ed.40}),
$$
|f_2(z_0)|\asymp K(|z_0|^2)^2,
$$
Hence,
$$
|f_2(z)|\asymp K^2,\ z\in \mathrm{neigh\,}(z_0).
$$
From this, (\ref{ed'.9}) and (\ref{ed.40}), we conclude first that
$\nabla_z|f_2|=\mO(K^3)$ and then that 
\begin{equation}\label{ed.44}
|\nabla _ze_2|={\mO}(K).
\end{equation}
This completes the proof of the fact that we can take $m=K$ above. In
particular (\ref{ed.41}) holds with $m=K(|z|^2)\asymp G(|z|)$, so 
\begin{equation}\label{ed.45}
\left| \sum_2^{N^2} \alpha _j\partial _ze_j \right| \le 
{\mO}(1)G|\alpha |\le {\mO}(1)GN,
\end{equation}
where we used the assumption that $|Q|\le C_1N$ in the last step. 

Combining this with (\ref{ed'.4}), (\ref{ed'.3}), (\ref{ed.22}), (\ref{ed.29}) and
the observation prior to Proposition \ref{ed1}, we get 
\[
\begin{split}
\partial _zf&=\frac{{\mO}(1)}{\delta G^2}\left(N|z|^{N-1} +\delta
  \left( \frac{|z|^N}{\delta G^2}+\delta N^2 \right)G^3+G^2\delta
  ^2NGN+\frac{G^2\delta ^2N^2}{r_0-|z|} \right)\\
&={\mO}(1)\left( \frac{N|z|^{N-1}}{\delta G^2}+\frac{|z|^N}{\delta
    G}+G\delta N^2+\frac{\delta N^2}{r_0-|z|} \right).
\end{split}
\]
In the last parenthesis the second term is dominated by the first one and
the third term is dominated by the fourth. If we
recall that $r_0-|z|\ge 1/N$, we get
\begin{equation}\label{ed.46}
\partial _zf={\mO}(1)\left(\frac{N|z|^{N-1}}{\delta G^2} +\delta N^3 \right).
\end{equation}
Similarly, from (\ref{ed'.2}), (\ref{ed.31}) we get
\[
\begin{split}
  \partial _{\overline{z}}f&=\frac{{\mO}(1)}{\delta
    G^2}\left(\delta \left(\frac{|z|^N}{\delta G^2}+\delta N^2 \right)
    G^3+G^2\delta ^2N\left(\left(\frac{|z|^N}{\delta G^2}+\delta
        N^2 \right)G^2+GN \right) \right)\\
&={\mO}(1)\left( \frac{|z|^N}{\delta G}+\delta N^2G +
	  N|z|^N +G^2\delta ^2N^3+G\delta N^2
\right).
\end{split}
\]
Using (\ref{ed.11}), we get
\begin{equation}\label{ed.47}
\partial _{\overline{z}}f={\mO}(1)\left(\frac{|z|^N}{\delta
    G}+\delta N^2G \right),
\end{equation}
see \eqref{ed.22}. This will be used together with the estimates $\partial _{\alpha
  _j}f={\mO}(\delta N)$ in (\ref{ed'.3}).

The differential form $dQ_1\wedge dQ_2\wedge ...\wedge dQ_{N^2}$ will
change only by a factor of modulus one if we express $Q$ in another
fixed orthonormal basis and we will choose for that the basis
$e_1(z_0),...,e_{N^2}(z_0)$:
\[
Q=\sum_1^{N^2} Q_k e_k(z_0),\ \ Q_k=(Q|e_k(z_0)).
\]
Write 
$$Q=\alpha _1\underbrace{\overline{Z}(z)}_{|Z(z)|e_1(z)}+\sum_2^{N^2} \alpha _ke_k(z)
$$
and restrict to $\alpha _1=f(z,\alpha _2,...,\alpha _{N^2})$, where we
sometimes identify $\alpha '\in \overline{Z}(z)^\perp$ with $(\alpha
_2,...,\alpha _{N^2})$:
$$
{{Q}_\vert}_{\alpha _1=f(z,\alpha ')}=
f(z,\alpha ')\overline{Z}(z)+\sum_2^{N^2} \alpha _ke_k(z).
$$
Then,
$$
Q_j=f(\overline{Z}(z)|e_j(z_0))+\sum_{k=2}^{N^2 }\alpha _k(e_k(z)|e_j(z_0)),
$$
\[
\begin{split}
dQ_j=&(d_zf+d_{\alpha
  '}f)(\overline{Z}(z)|e_j(z_0))+f(d_z\overline{Z}(z)|e_j(z_0))\\
&+\sum_{k=2}^{N^2}\alpha _k(d_ze_k(z)|e_j(z_0)) +\sum_{k=2}^{N^2}d\alpha _k(e_k(z)|e_j(z_0)).
\end{split}
\]
Taking $z=z_0$ until further notice, we get with $\alpha '=(\alpha
_2,...,\alpha _{N^2})$:
$$
dQ_j=(d_zf+d_{\alpha '}f)(\overline{Z}|e_j)+f(\overline{\partial
  _zZ}|e_j)d\overline{z}+\alpha _2(d_ze_2|e_j)+\begin{cases}d\alpha
  _j,\ j\ge 2,\\ 0,\ j=1\end{cases}.
$$
Here, we used (\ref{ed'.8}). The first term to the right is equal to
$(d_zf+d_{\alpha '}f)|\overline{Z}|$ when $j=1$ and it vanishes when
$j\ge 2$. The second term vanishes for $j\ge 3$, by (\ref{ed'.7}). The
third term is equal to $-\alpha _2(e_2|d_ze_j)$ (by differentiation of
the identity $(e_2|e_j)=\delta _{2,j}$) and it vanishes for $j\ge 3$
(remember that we take $z=z_0$). Thus, for $z=z_0$:
\[
\begin{split}
  dQ_1&=|\overline{Z}| (d_zf+d_{\alpha '}f)+f(\overline{\partial
    _zZ}|e_1)d\overline{z}-\alpha _2(e_2|d_ze_1),\\
  dQ_2&=d\alpha _2+f(\overline{\partial _zZ}|e_2)d\overline{z}-\alpha
  _2(e_2|d_ze_2),\\
dQ_j&=d\alpha _j,\ j\ge 3.
\end{split}
\]
When forming $dQ_1\wedge d\overline{Q}_1\wedge ...\wedge
dQ_{N^2}\wedge d\overline{Q}_{N^2}$ we see that the terms in $d\alpha
_j$ for $j\ge 3$ in the expression for $dQ_1$ will not contribute, so in that expression
we can replace $d_{\alpha '}f$ by $\partial _{\alpha _2}fd\alpha
_2$. Using (\ref{ed.46}), (\ref{ed.47}), (\ref{ed'.3}), (\ref{ed.22}),
(\ref{ed.31}) this gives, where ``$\equiv$'' means equivalence up to
terms that do not influence the $2{N^2}$ form above:
\[\begin{split}
dQ_1\equiv &-\alpha _2(e_2|d_ze_1) 
+{\mO}(1)\left(\frac{N|z|^{N-1}}{\delta G}+G\delta N^3 \right)dz\\
& +{\mO}(1)\left(\frac{|z|^N}{\delta }+G^2\delta N^2
\right)d\overline{z}+{\mO}(\delta NG)d\alpha _2.
\end{split}
\]
Similarly, using also (\ref{ed.44}),
\[
\begin{split}
dQ_2=d\alpha _2+{\mO}\left(\frac{|z|^N}{\delta }+\delta
  N^2G^2+|\alpha _2|G \right)d\overline{z}+{\mO}\left(|\alpha _2|G \right)dz.
\end{split}
\]
\par 
When computing $dQ_1\wedge dQ_2$ we notice that the terms in $dz
\wedge d\overline{z}$ will not contribute to the $2N^2$-form
$dQ_1\wedge d\overline{Q}_1\wedge ...\wedge dQ_{N^2}\wedge
d\overline{Q}_{N^2}$. We get 
\begin{equation}\label{ed.48}
\begin{split}
dQ_1\wedge d Q_2\equiv & -\alpha _2(e_2|d_ze_1)\wedge d\alpha _2 \\
& +{\mO}(1)\left(\frac{N|z|^{N-1}}{\delta G}+G\delta N^3+|\alpha
  _2|\delta NG^2 \right)dz\wedge d\alpha _2\\
&+ {\mO}(1)\left(\frac{|z|^N}{\delta }+G^2\delta N^2+|\alpha
  _2|\delta NG^2 \right)d\overline{z}\wedge d\alpha _2 .
\end{split}
\end{equation}
Here,
\[
\begin{split}
(e_2|d_ze_1)&=\left( e_2|d_z \left( |Z|^{-1} \right)\overline{Z}\right)
=\left(e_2||Z|^{-1}d_z\overline{Z} \right)+\left(e_2|
  d_z\left(|Z|^{-1} \right) \overline{Z} \right)\\
&=|Z|^{-1}\left(e_2 |\overline{\partial _zZ} d\overline{z} \right)
+0=|Z|^{-1} (e_2| \overline{\partial _zZ}) dz,
\end{split}
\]
so the first term in (\ref{ed.48}) is equal to 
$$
-\frac{\alpha _2}{|Z|}(e_2|\overline{\partial
  _zZ})dz\wedge d\alpha _2 ={\mO}(1)\alpha _2 G
d z\wedge d\alpha _2.
$$
Notice that 
$dQ_1\wedge d\overline{Q}_1 \wedge dQ_2 \wedge d\overline{Q}_2=-dQ_1
\wedge dQ_2 \wedge d\overline{Q}_1 \wedge d\overline{Q}_2$. From
(\ref{ed.48}) and its complex conjugate we get 
\[\begin{split}
&dQ_1\wedge d\overline{Q}_1\wedge dQ_2\wedge d\overline{Q}_2\\
&\equiv\left( - \frac{|\alpha _2|^2}{|Z|^2}\left| \left(e_2 |
      \overline{\partial _zZ} \right) \right|^2
      +{\mO}(1)\left(\frac{N|z|^{N-1}}{\delta G}+G\delta N^3+|\alpha
    _2|\delta NG^2 \right)^2\right.\\
& \left. +{\mO}(1)|\alpha _2|G\left(\frac{N|z|^{N-1}}{\delta
      G}+G\delta N^3+|\alpha _2|G^2\delta N \right)
 \right)dz\wedge d\overline{z}\wedge d\alpha _2\wedge
 d\overline{\alpha }_2 .
\end{split}
\]
\begin{prop}\label{ed4}
  We express $Q$ in the canonical basis in ${\C}^{N^2}$ or in any
  other fixed orthonormal basis. Let $e_1(z),...,e_{N^2}(z)$ be an
  orthonormal basis in ${\C}^{N^2}$ depending smoothly on $z$ and
  with $e_1(z)=|Z(z)|^{-1}\overline{Z}(z)$, ${\C}e_1(z)\oplus{\C}
  e_2(z)={\C}\overline{Z}(z)\oplus \overline{\partial
    _zZ}(z)$. Write $Q=\alpha _1\overline{Z}(z)+\sum_2^{N^2}\alpha
  _je_j(z)$, and recall that the hypersurface 
  \begin{equation*}
   \{(z,Q)\in D(0,r_0-1/N)\times B(0,C_1N);\, E_{-+}^\delta (z,Q)=0\}
  \end{equation*}
  is given by (\ref{ed.21}) with $f$ as in (\ref{ed.22}). Then the
restriction of $dQ\wedge d\overline{Q}$ to this hypersurface, is given
by 
\begin{equation}\label{ed.49}
\begin{split}
&dQ\wedge d\overline{Q}=J(f)dz\wedge d\overline{z}\wedge d\alpha '
\wedge d\overline{\alpha }',\\
&J(f)= - \frac{|\alpha _2|^2}{|Z|^2}\left| \left(e_2 |
      \overline{\partial _zZ} \right) \right|^2
      +{\mO}(1)\left(\frac{N|z|^{N-1}}{\delta G}+G\delta N^3+|\alpha
    _2|\delta NG^2 \right)^2\\
&\phantom{J(f)=}  +{\mO}(1)|\alpha _2|G\left(\frac{N|z|^{N-1}}{\delta
      G}+G\delta N^3+|\alpha _2|G^2\delta N \right).
\end{split}
\end{equation}
Here $\alpha '=(\alpha _2,...,\alpha _{N^2})$, $d\alpha '\wedge
d\overline{\alpha }'=d\alpha _2\wedge d\overline{\alpha }_2\wedge
...\wedge d\alpha _{N^2}\wedge d\overline{\alpha }_{N^2}$.
\end{prop}
\section{Proof of Theorem \ref{ed5}}\label{pfThm}
Let $Q\in {\C}^{N^2}$ be an $N\times N$ matrix whose entries are
independent random variables $\sim {\mathcal{N}}_{\C}(0,1)$, so that the
corresponding probability measure is
$$
\pi ^{-N^2}e^{-|Q|^2}(2i)^{-N^2}d\overline{Q} _1\wedge dQ_1\wedge
...\wedge d\overline{Q}_{N^2}\wedge dQ_{N^2}=\frac{1}{(2\pi
  i)^{N^2}}e^{-|Q|^2}d\overline{Q}\wedge dQ.
$$
We are interested in 
\begin{equation}\label{ed.50}
K_\phi ={\E}\left(1_{B_{{\C}^{N^2}}}(0,1)\sum _{\lambda \in
    \sigma (A_0+\delta Q)}\phi (\lambda ) \right),\ \ \phi \in C_0(D(0,r_0-1/N),
\end{equation}
which is of the form (\ref{ed.3}) with
\begin{equation}\label{ed.51}
m(Q)=1_{B_{{\C}^{N^2}}}(Q)\pi ^{-N^2}e^{-|Q|^2},
\end{equation}
so we have (\ref{ed.8}), (\ref{ed.9}) with $J(f)$ as in (\ref{ed.49})
and $f$ as in (\ref{ed.21}). More explicitly,
$$
\Xi (z)=\int_{|f|^2|Z(z)|^2+|\alpha '|^2\le C_1^2N^2} \pi
^{-N^2}e^{-|f(z,\alpha ')|^2|Z(z)|^2-|\alpha '|^2}J(f)(z,\alpha ')L(d\alpha ').
$$
By \eqref{ed.22}, \eqref{ed.11}, \eqref{ed.16} : 
\begin{equation*}
 |f| \leq \mO(1)\frac{N}{G}\left( \frac{|z|^N}{\delta NG} + \delta NG\right)
  \ll \frac{N}{G}.
\end{equation*}
We now strengthen \eqref{ed.11}, \eqref{ed.16} to the assumption
\begin{equation}\label{ed.52}
\frac{|z|^N}{\delta N G}+\delta NG \ll \frac{1}{N},\hbox{ for all }z\in D(0,r_0),
\end{equation}
implying that $|f|G \ll 1$, for all $z\in D(0,r_0)$. Equivalently, 
by the same reasoning as after \eqref{ed.25}, $r_0$ should satisfy 
\begin{equation}\label{ed.52.5}
 \frac{r_0^N}{\delta N G(r_0)} + \delta NG(r_0) \ll \frac{1}{N}.
\end{equation}
Then 
$$
e^{-|f(z,\alpha ')|^2|Z(z)|^2}
  = 1+{\mO}(1)N^2
    \left(\frac{|z|^N}{\delta NG}+\delta NG\right)^2,
$$
and using (\ref{ed.49}), we get
\[
\begin{split}
\Xi (z)=&\left(1+ {\mO}(1) N^2\left(\frac{|z|^N}{\delta NG}
  +\delta NG\right)^2 \right)\times \\ & \frac{|(e_2|\overline{\partial
_zZ})|^2}{|Z|^2}\int _{|(f|Z|,\alpha ')|\le C_1N}|\alpha _2|^2
e^{-|\alpha '|^2}\pi ^{1-N^2}L(d\alpha ')\\
&+{\mO}(1)\int e^{-|\alpha '|^2}\left(\frac{N|z|^{N-1}}{\delta
    G}+G\delta N^3+|\alpha _2|\delta NG^2
\right)^2\pi^{1-N^2}L(d\alpha ')\\
+&{\mO}(1)\int e^{-|\alpha '|^2}|\alpha _2|G\left(\frac{N|z|^{N-1}}{\delta
    G}+G\delta N^3+|\alpha _2|\delta NG^2
\right)\pi^{1-N^2}L(d\alpha ').
\end{split}
\]

\par Since $|f||Z|\ll N$, the first integral is equal to
$$
\int_{\C} \frac{1}{\pi }|w|^2e^{-|w|^2}L(dw)+{\mO}
\left(e^{-N^2/{\mO}(1)} \right) =1+{\mO}\left(e^{-N^2/{\mO}(1)} \right).
$$
The sum of the other two integrals is equal to 
\[
\begin{split}
&{\mO}(1)\left(\left(\frac{N|z|^{N-1}}{\delta G}+G\delta N^3+\delta
  NG^2\right)^2+G \left(\frac{N|z|^{N-1}}{\delta G}+G\delta N^3+\delta
  NG^2\right)\right)\\
&={\mO}(1)\left(\left(\frac{N|z|^{N-1}}{\delta G}+G\delta
    N^3\right)^2+G \left(\frac{N|z|^{N-1}}{\delta G}+G\delta
    N^3\right)\right) .
\end{split}
\]
Noticing that
$$
\frac{|(e_2|\overline{\partial _zZ})|^2}{|Z|^2} ={\mO}(G^2),
$$
we deduce that
\begin{equation}\label{ed.53}
\begin{split}
\Xi(z)=& \frac{|(e_2|\overline{\partial _zZ})|^2}{|Z|^2} \\
&+ {\mO}(1)\left(G^2N^2\left(\frac{|z|^{N-1}}{\delta G^2}+\delta
    N^2\right)^2+G^2N \left(\frac{|z|^{N-1}}{\delta G^2}+\delta
    N^2\right)\right) .
\end{split}
\end{equation}
We next study the leading term in (\ref{ed.53}), given by
\begin{equation}\label{ed.54}
\frac{|(\overline{\partial _zZ}|e_2)|^2}{\pi |Z|^2}.
\end{equation}
Since $\overline{\partial _zZ}$ belongs to the span of
$e_1=\overline{\partial _zZ}/|Z|$ and $e_2$, we have
$$
|(\overline{\partial _zZ}|e_2)|^2=|\overline{\partial
  _zZ}|^2-|(\overline{\partial _zZ}|e_1)|^2,
$$ 
so the leading term (\ref{ed.54}) is 
$$\frac{1}{\pi |Z|^2}\left(|\overline{\partial
    _zZ}|^2-\frac{|(\overline{\partial _zZ}|\overline{Z})|^2}{|Z|^2}
\right),$$
which by (\ref{ed.36}) is equal to 
\begin{equation}\label{ed.55}
{\frac{2}{\pi t}\left(\frac{(t\partial _t)^2K}{K}-\frac{(t\partial _tK)^2}{K^2} \right)}_{t=|z|^2}.
\end{equation}
Here, $K=K_N(t)=\sum_0^{N-1}t^\nu $ is the function appearing in
Proposition \ref{ed3}. Let us first compute the limiting quantity
obtained by replacing $K=K_N$ in (\ref{ed.55}) by $K_\infty
=1/(1-t)$. Since $\partial _tK_\infty =K_\infty ^2$, we get
$$
t\partial _tK_\infty =tK_\infty ^2,\ \ (t\partial _t)^2K_\infty
=tK_\infty ^2+2t^2K_\infty ^3,
$$
and
\begin{equation}\label{ed.56}
\frac{2}{\pi t}\left(\frac{(t\partial _t)^2K_\infty }{K_\infty
  }-\frac{(t\partial _tK_\infty )^2}{K_\infty ^2} \right)=\frac{2}{\pi
}K_\infty ^2=\frac{2}{\pi }\frac{1}{(1-t)^2}.
\end{equation}

\par We next approximate the expression (\ref{ed.55}) with
(\ref{ed.56}), using (\ref{ed.39}) and the fact that $K=(1+{\mO}(t^N))K_\infty $ 
(uniformly with respect to $N$). The expression
(\ref{ed.55}) is equal to
\[\begin{split}
&\frac{2}{\pi tK^2}(K(t\partial _t)^2K-(t\partial _tK)^2)\\
&= \frac{2(1+{\mO}(t^N))}{\pi tK_\infty ^2}\left(K_\infty
  (t\partial _t)^2K_\infty -(t\partial _tK_\infty )^2 +{\mO}
  (N^2t^NK_\infty ^2) \right).
\end{split}
\]
Here,
$$
(t\partial _tK_\infty )^2={\mO}(t^2K_\infty ^4),\ \ K_\infty
(t\partial _t)^2K_\infty ={\mO}(tK_\infty ^3+t^2K_\infty ^4),
$$
so the last expression becomes,
$$
\frac{2}{\pi t}\left(\frac{(t\partial _t)^2K_\infty }{K_\infty
  }-\frac{(t\partial _tK_\infty )^2}{K_\infty ^2} \right)+{\mO}
  (t^NK_\infty +t^{N+1}K_\infty ^2+t^{N-1}N^2),
$$
where the first two terms in the remainder are dominated by the last one.
We conclude that the difference
between the expressions (\ref{ed.55}) and (\ref{ed.56}) is ${\mO}(t^{N-1}N^2)$, 
and using also (\ref{ed.53}), we get,
\begin{equation}\label{ed.57}
\begin{split}
\Xi(z)=& \frac{2}{\pi (1-|z|^2)^2}+{\mO}(|z|^{2(N-1)}N^2) \\
&+ {\mO}(1)\left(G^2N^2\left(\frac{|z|^{N-1}}{\delta G^2}+\delta
    N^2\right)^2+G^2N \left(\frac{|z|^{N-1}}{\delta G^2}+\delta
    N^2\right)\right) .
\end{split}
\end{equation}
The remainder term can be written
\[
{\mO}(G^2)\left(\frac{|z|^{2(N-1)}N^2}{G^2}+\frac{|z|^{2(N-1)}N^2}{\delta
  ^2G^4}+\delta ^2N^6 +\frac{|z|^{N-1}N}{\delta G^2}+\delta N^3 \right).
\]
By (\ref{ed.52}), $\frac{1}{\delta G}\gg N^2$, so the second term is 
$$
\gg \frac{|z|^{2(N-1)}N^2}{G^2}N^4,
$$
which is much larger than the first term. We now strengthen
(\ref{ed.52}) to
$$
\frac{|z|^{N-1}}{\delta G^2}+\delta N^2\ll \frac{1}{N},
$$
or equivalently to
\begin{equation}\label{ed.58}
\frac{|z|^{N-1}N}{\delta G^2}+\delta N^3\ll 1 .
\end{equation}
Then remainder in (\ref{ed.57}) becomes
$$
{\mO}(G^2)\left(\frac{|z|^{N-1}N}{\delta G^2}+\delta N^3 \right),
$$
and (\ref{ed.57}) becomes
\begin{equation}\label{ed.59}
\Xi (z)=\frac{2}{\pi (1-|z|^2)^2}\left(1+{\mO}
\left(\frac{|z|^{N-1}N}{\delta G^2} +\delta N^3 \right) \right),
\end{equation}
which concludes the proof of Theorem \ref{ed5}.

\providecommand{\bysame}{\leavevmode\hbox to3em{\hrulefill}\thinspace}
\providecommand{\MR}{\relax\ifhmode\unskip\space\fi MR }
% \MRhref is called by the amsart/book/proc definition of \MR.
\providecommand{\MRhref}[2]{%
  \href{http://www.ams.org/mathscinet-getitem?mr=#1}{#2}
}
\providecommand{\href}[2]{#2}


\begin{thebibliography}{10}

\bibitem{BM}
W.~Bordeaux-Montrieux, \emph{{Loi de Weyl presque s{\^u}re et r{\'e}solvent
  pour des op{\'e}rateurs diff{\'e}rentiels non-autoadjoints, Th{\'e}se}},
  pastel.archives-ouvertes.fr/docs/00/50/12/81/PDF/manuscrit.pdf (2008).

\bibitem{Da97}
E.~B. Davies, \emph{{Pseudospectra of Differential Operators}}, J. Oper. Th
  \textbf{43} (1997), 243--262.

\bibitem{Da99}
E.B. Davies, \emph{{Pseudo{\textendash}spectra, the harmonic oscillator and
  complex resonances}}, Proc. of the Royal Soc.of London A \textbf{455} (1999),
  no.~1982, 585--599.

\bibitem{DaHa09}
E.B. Davies and M.~Hager, \emph{{Perturbations of Jordan matrices}}, J. Approx.
  Theory \textbf{156} (2009), no.~1, 82--94.

\bibitem{NSjZw04}
N.~Dencker, J.~Sj{\"o}strand, and M.~Zworski, \emph{{Pseudospectra of
  semiclassical (pseudo-) differential operators}}, Communications on Pure and
  Applied Mathematics \textbf{57} (2004), no.~3, 384--415.

\bibitem{GuMaZe14}
A.~Guionnet, P.~Matchett Wood, and {0. Zeitouni}, \emph{{Convergence of the
  spectral measure of non-normal matrices}}, Proc.~AMS \textbf{142} (2014),
  no.~2, 667--679.

\bibitem{Ha06b}
M.~Hager, \emph{{Instabilit{\'e} Spectrale Semiclassique d{\rq}Op{\'e}rateurs
  Non-Autoadjoints II}}, Annales Henri Poincare \textbf{7} (2006), 1035--1064.

\bibitem{Ha06}
\bysame, \emph{{Instabilit{\'e} spectrale semiclassique pour des op{\'e}rateurs
  non-autoadjoints I: un mod{\`e}le}}, Annales de la facult{\'e} des sciences
  de Toulouse S{\'e}. 6 \textbf{15} (2006), no.~2, 243--280.

\bibitem{HaSj08}
M.~Hager and J.~Sj{\"o}strand, \emph{{Eigenvalue asymptotics for randomly
  perturbed non-selfadjoint operators}}, Mathematische Annalen \textbf{342}
  (2008), 177--243.

\bibitem{HoKrPeVi09}
J.~Hough, M.~Krishnapur, Y.~Peres, and B.~Vir{\'a}g, \emph{{Zeros of Gaussian
  Analytic Functions and Determinantal Point Processes}}, American Mathematical
  Society, 2009.

\bibitem{Kac43}
M.~Kac, \emph{{On the average number of real roots of a random algebraic
  equation}}, Bulletin of the American Mathematical Society \textbf{49} (1943),
  no.~4, 314--320.

\bibitem{Sh08}
B.~Shiffman, \emph{{Convergence of random zeros on complex manifolds}}, Science
  in China Series A: Mathematics \textbf{51} (2008), 707--720.

\bibitem{ShZe98}
B.~Shiffman and S.~Zelditch, \emph{{Distribution of Zeros of Random and Quantum
  Chaotic Sections of Positive Line Bundles}}, Communications in Mathematical
  Physics \textbf{200} (1999), 661--683.

\bibitem{SZ03}
\bysame, \emph{{Equilibrium distribution of zeros of random polynomials}}, Int.
  Math. Res. Not. (2003), 25--49.

\bibitem{SjAX1002}
J.~Sj{\"o}strand, \emph{{Spectral properties of non-self-adjoint operators}},
  (2009).

\bibitem{SjZw07}
J.~Sj{\"o}strand and M.~Zworski, \emph{{Elementary linear algebra for advanced
  spectral problems}}, Annales de l'Institute Fourier \textbf{57} (2007),
  2095--2141.

\bibitem{So00}
M.~Sodin, \emph{{Zeros of Gaussian Analytic Functions and Determinantal Point
  Processes}}, Mathematical Research Letters (2000), no.~7, 371--381.

\bibitem{SoTs04}
M.~Sodin and B.~Tsirelson, \emph{{Random complex zeroes, I. Asymptotic
  normality}}, Israel Journal of Mathematics \textbf{144} (2004), 125--149.

\bibitem{TrEm05}
L.~N. Trefethen and M.~Embree, \emph{{Spectra and Pseudospectra: The Behavior
  of Nonnormal Matrices and Operators}}, Princeton University Press, 2005.

\bibitem{Vo14}
M.~Vogel, \emph{{The precise shape of the eigenvalue intensity for a class of
  non-selfadjoint operators under random perturbations}}, arxiv:1401.8134v1
  [math.SP] (2014).

\bibitem{Zw02}
M.~Zworski, \emph{Numerical linear algebra and solvability of partial
  differential equations,} Comm. Math. Phys. 229(2)(2002), 293–307.

\bibitem{ZwChrist10}
M.~Zworski and T.J. Christiansen, \emph{{Probabilistic Weyl Laws for Quantized
  Tori}}, Communications in Mathematical Physics \textbf{299} (2010).

\end{thebibliography}
\end{document}